\documentclass[a4paper,reqno]{amsart}

\usepackage{amssymb}
\usepackage[utf8]{inputenc}
\usepackage[T1]{fontenc}
\usepackage{graphicx,color}
\usepackage{enumerate}

\newcommand{\ci}{\operatorname{ci}}
\newcommand{\real}{\mathbb{R}}

\newcommand{\chf}{\mathbf{1}}
\newcommand{\Bdelta}{\delta^B}

\newcommand{\aff}{\operatorname{aff}}
\newcommand{\vol}{\operatorname{vol}}

\newcommand{\cS}{\mathcal{S}}
\newcommand{\cT}{\mathcal{T}}
\newcommand{\cK}{\mathcal{K}}
\newcommand{\cR}{\mathcal{R}}
\newcommand{\dotvar}{\,\cdot\,}
\newcommand{\perim}{\operatorname{per}}
\newcommand{\bd}{\operatorname{bd}}
\newcommand{\eps}{\varepsilon}

\newcommand{\len}{\operatorname{len}}
\newcommand{\length}{\operatorname{len}}
\newcommand{\intr}{\operatorname{int}}
\newcommand{\relint}{\operatorname{relint}}
\newcommand{\cl}{\operatorname{cl}}
\newcommand{\conv}{\operatorname{conv}}
\newcommand{\usphere}{\mathbb{S}}
\newcommand{\ucircle}{\mathbb{S}^1}
\newcommand{\dd}{\operatorname{d}}
\newcommand{\thmheader}[1]{{\upshape (#1)}}
\newcommand{\sprod}[2]{\left<#1,#2\right>}
\newcommand{\pairing}[2]{\left(#1,#2\right)}
\newcommand{\setcond}[2]{\left\{ #1 \,:\, #2\right\}}
\newcommand{\supp}{\operatorname{supp}}
\newcommand{\pderiv}[1]
{
	\mathop{\frac{\partial}{\partial #1}}
}

\newcommand{\pderivleft}[1]{\mathop{\frac{\partial^{-}}{\partial #1}}}
\newcommand{\rderiv}[1]{\mathop{\partial^+ #1}}

\newcommand{\wid}{w}

\newcommand{\Prob}{\mathop{\mathrm{Prob}}}
\newcommand{\wcov}[1]{g_{{#1,w}}}
\newcommand{\core}{\operatorname{core}}
\newcommand{\thecap}{\operatorname{cap}}
\newcommand{\arc}{\operatorname{arc}}

\newcommand{\cH}{\mathcal{H}}
\newcommand{\gperim}[1]{g_{#1,\perim_B}}
\newcommand{\gphi}[1]{g_{#1,\phi}}

\newcommand{\inspar}{\operatorname{ip}}

\newcommand{\EvalFor}[2]{
	\left. #1 \right|_{#2}
}

\newcommand{\PartialDerivAt}[3]
{ 
	\EvalFor{
		\pderiv{#2}
		#1 
	}{#2  = #3}
}

\newcommand{\LeftPartialDerivAt}[3]
{ 
	\EvalFor{
		\pderivleft{#2}
		#1 
	}{#2  = #3}
}

\newtheorem{theorem}{Theorem}[section]

\newtheorem{lemma}[theorem]{Lemma}
\newtheorem{proposition}[theorem]{Proposition}
\newtheorem{claim}{Claim}[theorem]
\theoremstyle{definition}

\newtheorem{remark}[theorem]{Remark}

\numberwithin{equation}{section}

\newenvironment{FigTab}[2]{
        \begin{figure}[htb]
        \setlength{\unitlength}{#2}
        \begin{center}
        \begin{tabular}{#1}
}{
    \end{tabular}
    \end{center}
    \end{figure}
}

\begin{document}

\title[Covariograms generated by valuations]{
	Covariograms
	generated by valuations
}
\author{
	Gennadiy Averkov
}
\address{
	Faculty of Mathematics, 
	University of Magdeburg, 
	Universit\"atsplatz 2, 
	D-39106 Magdeburg, 
	Germany
}
\email{
	averkov@math.uni-magdeburg.de
}
\author{
	Gabriele Bianchi
}
\address{
	Dipartimento di Matematica e Informatica ``U. Dini'', 
	Universit\`a di Firenze, 
	Viale Morgagni 67/A,
	I-50134 Firenze, Italy
}
\email{
	gabriele.bianchi@unifi.it
}

\thanks{The authors have been supported by the Gruppo
Nazionale per l’Analisi Matematica, la Probabilit\`a e le loro
Applicazioni (GNAMPA) of the Istituto Nazionale di Alta Matematica
(INdAM)}


\begin{abstract}
	Let \( \phi \) be a real-valued valuation on the family of compact convex subsets of $\real^n$ and let \( K \) be a convex body in \( \real^n \). We introduce the \( \phi \)-covariogram \( \gphi{K} \) of \( K \) as the function associating to each \( x \in \real^n \) the value \( \phi(K \cap (K+x)) \).
	If \( \phi \) is the volume, then \( \gphi{K} \) is the covariogram, extensively studied in various sources. When \( \phi \) is a quermassintegral (e.g., surface area or mean width) \( \gphi{K} \) has been introduced by Nagel \cite{nagel_habil}.

	We study various properties of \( \phi \)-covariograms, mostly in the case \( n=2 \) and under the assumption that \( \phi \) is  translation invariant, monotone and even.
	We also consider the generalization of Matheron's covariogram problem to the case of \( \phi \)-covariograms, that is, the problem of determining an unknown convex body \( K \), up to translations and point reflections, by the knowledge of \( \gphi{K} \).
	A positive solution to this problem is provided under different assumptions, including the case that \( K \) is a polygon and \( \phi \) is either strictly monotone or $\phi$ is the width in a given direction.
	We prove that there are examples in every dimension $n\geq3$ where $K$ is determined by its covariogram but it is not determined by its width-covariogram.
	We also present some consequence of this study in stochastic geometry.
\end{abstract}

\maketitle

\section{Introduction}

Let \( K \) be a convex body in \( \real^n \). The covariogram of \( K \) is the function \( g_K \) which associates to each \( x \in \real^n \) the volume of \( K \cap (K+x) \):
\begin{equation*}
	g_K(x) := \vol\left(K \cap (K+x)\right).
\end{equation*}
The data provided by \( g_K(x) \) can be interpreted in several ways within different contexts, using purely geometric, functional-analytic and probabilistic terminology. 
As a result, covariograms of convex bodies and other sets appear naturally in various research areas including convex geometry, image analysis, geometric shape and pattern matching, phase retrieval in Fourier analysis, crystallography and geometric probability. See Baake and Grimm~\cite{BaakeGrimm}, Bianchi, Gardner and Kiderlen~\cite{BiGaKi11} and references therein, Matheron~\cite{MR0385969} and Schymura \cite{Schymura-2011}.

The notion of volume can be naturally extended to the notion of valuation. (See Section~\ref{basic:sect} for all unexplained definitions.) Let \( \cK^n \) be the family of all compact, convex subsets of \( \real^n \) and let \( \phi : \cK^n \rightarrow \real \) be a valuation. We introduce the \( \phi \)-covariogram of \( K \) as the function \( \gphi{K} : \real^n \rightarrow \real \) defined for $x\in\real^n$ by
\begin{equation*} 
	\gphi{K}(x) := \phi(K \cap (K+x)).
\end{equation*}

Werner Nagel  in his Habilitationsschrift~\cite[pp.~68-69]{nagel_habil}
introduces \( \gphi{K} \) in the case that \( \phi \) is an arbitrary quermassintegral (this includes the case of volume, surface area and mean width). Gardner \& Zhang \cite[p.~524]{MR1623396} suggests to generalize \( g_K \) substituting the volume with an arbitrary log-concave measure in $\real^n$.
The $\phi$-covariogram appears naturally in some problems in stochastic geometry.
See later in the introduction for more on this point. 

We assume that $\phi$ belongs to the class $\Phi^n$ of real-valued, even, translation invariant valuations on $\cK^n$ which are monotone with respect to inclusion and which vanish on singletons.
The covariogram $g_K$ is clearly unchanged by a translation or a reflection of $K$ (the term reflection will always mean reflection at a point) and the assumption that $\phi$ is even and translation invariant forces $\gphi{K}$ to  maintain these   invariance properties.
The assumption that $\phi$ vanishes on singletons is not restrictive, as explained in Section~\ref{basic:sect}.

Most results in this paper are in the plane. Every $\phi\in\Phi^2$ can be decomposed in an unique way as
\begin{equation}\label{representation_valuation_perim}
 \phi(K)=\perim_B(K)+\alpha\vol(K),\quad \text{for each $K\in\cK^2$,}
\end{equation}
for a suitable $\alpha\geq0$ and an $o$-symmetric closed convex set $B$ with $o\in\intr B$ (see Theorem~\ref{repr phi thm}).
Here $\perim_B$ denotes the perimeter with respect to the seminorm associated to the unit ball $B$.
An alternative equivalent representation is
\begin{equation}\label{representation_valuation_mixedvol}
 \phi(K)=V(K,H)+\alpha\vol(K),\quad \text{for each $K\in\cK^2$,}
\end{equation}
where $H\in\cK^2$ is $o$-symmetric and nonempty and $V(K,H)$ denotes mixed area. A consequence of \eqref{representation_valuation_perim} is that for every planar convex body $K$ we have
\begin{equation} \label{decomposition_of_phi-cov}
        \gphi{K}  = \gperim{K} + \alpha g_{K}.
\end{equation} We call \( \gperim{K} \) the \emph{perimeter-covariogram}. When $B=\real^2$, the function $\gperim{K}$ vanishes and then $\gphi{K}=\alpha g_K$. When $B$ is the Euclidean unit ball, $\gperim{K}(x)$ is the usual Euclidean perimeter of $K\cap(K+x)$. When $B$ is the strip $\{x\in\real^2 : |\sprod{x}{z}|\leq 1\}$, for some $z\in\ucircle$, then $\gperim{K}(x)$ coincides with twice the width of $K\cap(K+x)$ with respect to $z$.

We study various aspects of \( \phi \)-covariograms, but the main part of the paper is devoted to the following problem.
\smallskip

\textbf{The $\phi$-covariogram problem.} \emph{Does the knowledge of $\phi$ and $\gphi{K}$ determine a
convex body $K$, within all convex bodies, up to translations and reflections?}
\smallskip

To make the statement of the above problem and the formulations of the following results precise, we clarify that we say that $K \in\ \cK^n$ is \emph{determined} by the knowledge of $\phi$ and $\gphi{K}$, within a family $\cH \subset \cK^n$, up to a group $\cT$ of transformations of $\real^n$ if the equality $\gphi{K} = \gphi{H}$ for $H \in \cH$ implies $K=T(H)$ for some $T \in \cT$.

The corresponding problem for the covariogram was posed by G.~Matheron in 1986 and has received much attention in recent years. Peter Gruber~\cite{gruber-private} suggested to study the $\phi$-covariogram problem in the case where $\phi$ is the Euclidean perimeter. We prove the following results.
\begin{theorem}
	\label{retrieval centr sym thm}
	Let \( \phi\in\Phi^2 \setminus \{0\} \) and let $K$ be a centrally symmetric planar convex body. Then $K$ is determined by the knowledge of $\phi$ and $g_{K,\phi}$,
	up to translations, within the class of all planar convex bodies.
\end{theorem}

Theorem~\ref{retrieval centr sym thm} asserts that the knowledge of $\phi \in \Phi^2 \setminus \{0\}$ and $g_{K,\phi}$ is sufficient for testing whether a given planar convex body $K$ is centrally symmetric or not. Once the symmetry of $K$ has been detected, the determination of $K$ by $g_{K,\phi}$ is trivial, since $2 K$ coincides with the support of $g_{K,\phi}$, up to translations.

We call $\phi \in \Phi^2 \setminus \{0\}$ \emph{strictly monotone} if for all $K, H \in \cK^2$ such that $K$ is a nonempty, proper subset of $H$ the strict inequality $\phi(K) < \phi(H)$ holds. For strictly monotone valuations we show the following.

\begin{theorem}
	\label{retrieval of polygons thm}
	Let $\phi \in \Phi^2 \setminus \{0\}$ be strictly monotone with respect to inclusion and let $P$ be a convex polygon. Then $P$ is determined by the knowledge of $\phi$ and of $g_{P,\phi}$, up to translations and reflections, within the class of all planar convex bodies.
\end{theorem}

A valuation $\phi\in\Phi^2$ written as in~\eqref{representation_valuation_perim} is strictly monotone with respect to inclusion if and only if either $\alpha>0$ or $\alpha=0$ and   $B$ is strictly convex (see Proposition~\ref{prop on B and per B}).
Thus Theorem~\ref{retrieval of polygons thm} applies also to the perimeter-covariogram corresponding to the standard Euclidean perimeter.

\begin{theorem}
	\label{retrieval for width-covar thm}
	Let $z\in\ucircle$, let $\phi$ be the width with respect to $z$ and let $P$ be a convex polygon. Then $P$ is determined by the knowledge of $\phi$ and of  $g_{P,\phi}$, up to translations and reflections, within the class of all planar convex bodies.
\end{theorem}

The answer to the volume-covariogram problem is positive for every planar convex body, it is positive for convex polytopes in $\real^3$ (see Bianchi~\cite{Bianchi-2009-polytopes}) but the case of a general convex body in $\real^3$ is still open, and there are examples of nondetermination, as well as positive results in some subclasses of the class of convex bodies,  in every dimension $n\geq4$ (see Goodey, Schneider and Weil~\cite{Goodey-Schneider-Weil-1997}, Bianchi~\cite{MR2108257} and \cite{Bia_FTinCn_13+}).
The proof of the positive answer in the plane is still divided in two papers, with Bianchi~\cite{MR2108257} dealing with convex bodies which are not strictly convex or whose boundary is not everywhere differentiable, and Averkov and Bianchi~\cite{averkov-bianchi-2009} dealing with the remaining more difficult cases. No unifying proof still exists. At the moment it appears out of reach proving a positive answer for the $\phi$-covariogram problem  for general planar convex bodies, and we have decided to study this problem mostly in the class of polygons, where some technical aspects are simpler to handle.
Note that the class of convex polytopes has a remarkable aspect. In all known situations where counterexamples of nondetermination by the covariogram  (as well as by the cross-covariogram \cite{BiCrosscov09}) exist, these examples can also be constructed as convex polytopes.
Furthermore, when $\phi$ is the volume, high smoothness of the boundary of the body seems to depose in favor of determination~\cite{Bia_FTinCn_13+}.

See the beginning of Section~\ref{sect retrieval} for a detailed description of the proofs of Theorems~\ref{retrieval centr sym thm}, \ref{retrieval of polygons thm} and~\ref{retrieval for width-covar thm}. Here we make only a few comments. The structure of the proof of Theorem~\ref{retrieval of polygons thm} is similar to that of the corresponding result for the volume-covariogram problem. One of the tools in this proof is the geometric interpretation of the radial derivative of the perimeter-covariogram proved in Theorem~\ref{radial deriv of g perim}.  We do not know whether the $\phi$-covariogram problem has a positive answer for \emph{every} $\phi\in\Phi^2$, when $K$ is a polygon, and Theorem~\ref{retrieval for width-covar thm} can be seen as a step in investigating this. We remark that the absence of strict monotonicity makes the proof of Theorem~\ref{retrieval for width-covar thm} much more involved compared to the proof of Theorem~\ref{retrieval of polygons thm}.

Section~\ref{subsection nonuniqueness} presents some counterexamples of nondetermination in dimension $n\geq3$. 
The construction leading to counterexamples for the covariogram in dimension $n\geq4$, can be generalized to the $\phi$-covariogram for every $\phi$ which is invariant with respect to the group of isometries of the Euclidean space $\real^n$.
The width-covariogram however presents some novelties which suggest that  it provides less information about the body than $g_K$. It exhibits counterexamples with a structure richer than that of the covariogram.
A consequence of this is that while the volume-covariogram problem has a positive answer for all convex polytopes in $\real^3$ as well as for every centrally symmetric convex body in any dimension, there are examples of centrally symmetric convex polytopes in $\real^n$, for every $n\geq3$, that are not determined by the width-covariogram.
\begin{theorem}\label{counterexamples_width_cov}
Let $z\in \usphere^{n-1}$, let $\phi$ be the width with respect to $z$ and let $n\geq 3$. There exist convex polytopes $K$, $K'$ in $\real^n$ such that $K$ is centrally symmetric, $K'$ is not a translation of $K$ and $\gphi{K}=\gphi{K'}$. 
\end{theorem}
Theorem~\ref{retrieval centr sym thm} cannot thus be extended in full generality to dimension $n\geq3$. 

Beside the $\phi$-covariogram problem, we also study the extension to this more general setting of two aspects of the covariogram  which, in our opinion, are among the most important, namely, its connection with stochastic geometry and its representation as a convolution.
The study of which information about a convex body $K$ can be inferred by the distribution of the length of a random chord of $K$ goes back to Blaschke~\cite[Section~4.2]{MR2162874}. When this distribution is provided separated direction by direction (i.e., for each $u\in\usphere^{n-1}$,  the distribution of the length of a random chord parallel to $u$ is given) its knowledge is equivalent to the knowledge of the $\phi$-covariogram of $K$, with $\phi$ depending on the type of randomness. The next result is an example of these connections.
\begin{theorem}
	\label{thm random lambda sect}
	Let \( B\) be an $o$-symmetric closed convex subset of $\real^2$ with $o \in \intr B$ and $B\neq\real^2$. Let \( K \in \cK_0^2\).  Let \( Y \) be a random variable  distributed in \( \bd K \) with density \(\len_B/ \perim_B(K)\) and, for $u\in\ucircle$,  let $L_{\gamma,u}$ denote the length of the chord of $K$ parallel to $u$ and passing through $Y$.
	Then the following holds:
	\begin{enumerate}[(I)]
		\item For every $u \in \ucircle$, the distribution of $L_{\gamma,u}$ is determined by $B$ and $\gperim{K}$.	
		Conversely, the knowledge of $B$ and of the distribution of \( L_{\gamma,u} \) for every $u \in \ucircle$ determines \( \gperim{K} \).
		\item If
			\begin{enumerate}[(a)]
				\item $K$ is centrally symmetric or
				\item \( K \) is a polygon and \( B \) is either strictly convex or a strip,
			\end{enumerate}
			then the knowledge of $B$
			and of the distribution of \( L_{\gamma,u} \) for all directions \( u\in \ucircle \) determines \( K \), up to translation and reflection, in the class of all planar convex bodies.
	\end{enumerate}
\end{theorem}
The random variable \( L_{\gamma,u} \) has been introduced by Ehlers and Enns \cite{EnnsEhlers-1981} when $B$ is the Euclidean ball.
See Theorem~\ref{from:rand:var:to:cov} for a similar result for different random variables.

The fact that the covariogram can be written as an autocorrelation, i.e. $g_K=\chf_K\ast \chf_{-K}$, has important consequences on its study. For instance it connects the covariogram to the phase retrieval problem and to some of the above mentioned problems in stochastic geometry.
The $\phi$-covariogram, with $\phi\in\Phi^2$, cannot be written as an autocorrelation but can be written as a convolution, with  formulas involving $\chf_K$ and a suitable measure supported on the boundary of $K$ (see Theorem~\ref{g phi thm}). We remark that it is not clear which $\phi$-covariograms, with $\phi\in\Phi^n$ and $n\geq3$, can be written as convolutions.

Let us give an overview of the structure of the manuscript. In Section~\ref{basic:sect} we collect the necessary background material on convex sets, norms and seminorms, distributions and valuations. In Section~\ref{sect:covariograms_as_convolutions} we study various global properties of \( \gphi{K} \) and represent \( \gphi{K} \) as a convolution. In Section~\ref{sect deriv} we determine a geometric meaning of the radial derivative of \( \gphi{K} \). Section~\ref{sect retrieval} is the longest one and is divided in four subsections. The first three contain respectively the proofs of Theorems~\ref{retrieval centr sym thm}, \ref{retrieval of polygons thm} and \ref{retrieval for width-covar thm}. The fourth one  contains the results regarding nondetermination, including the proof of Theorem~\ref{counterexamples_width_cov}. Section~\ref{sect rand var} is devoted to the connections between the $\phi$-covariogram and stochastic geometry. In Section~\ref{sect:open_problems} we present various open problems and possible directions of further research.

\section{Notations and background material} \label{basic:sect}

\subsection{General notations for \( \real^n \)} 
The origin of \( \real^n \) is denoted by \( o \). By \( \sprod{\dotvar}{\dotvar} \) we denote the standard Euclidean product in \( \real^n \) and by \( \|\dotvar\| \) the corresponding norm.  The unit sphere in \( \real^n \) centered at \( o \) is denoted by \( \usphere^{n-1}. \) For \( u \in \real^n \setminus \{o\} \), by \( l_u \) we denote the line through \( o \) parallel to \( u \) (i.e., the linear span of \( \{u\} \)). For \( a, b \in \real^n \) by \( [a,b] \) we denote the line segment joining \( a \) and \( b \).

When $n=2$,  $\cR$
denotes the linear operation of rotation by 90 degrees around the origin in counterclockwise order. Let $A\subset \real^n$. The boundary, closure and interior of $A$ are abbreviated by \( \bd A\), \( \cl A \) and \( \intr A\), respectively. We denote by $DA$ the set
\[
DA : = \{x-y : x,y\in A\}.                                                                                                                                                                                                                                                                                                                                                                                                                                                                                                                                                                                   \]
We call $DA$ the \emph{difference set} of $A$. By \( \chf_A \) we denote the characteristic function of \( A \), that is, the function equal to \( 1 \) on \( A \) and equal to \( 0 \) on the complement of \( A \).

By \( \vol \) we denote the volume in \( \real^n \), that is, the Lebesgue measure in \( \real^n \).  The integrals of the form \( \int_{\real^n} f(x) \dd x \) for functions \( f : \real^n \rightarrow \real \) are assumed to be defined with respect to the Lebesgue measure in \( \real^n \).

\subsection{Convex geometry} 
	\label{subsect convex geom}
	
 By \( \cK^n \) we denote the set of all compact convex subsets of \( \real^n \) and by \( \cK_0^n \) the set of all convex bodies in \( \real^n \), that is, compact convex subsets of \(  \real^n  \) having nonempty interior.  For background information on convex sets we refer to \cite{Schneider-1993}.
 By \( \conv A\) we denote the \emph{convex hull} of $A$.  For \( K \in\cK_0^n\) the difference set \( DK \) is a convex body, called the \emph{difference body} of \( K \). 

If \( u \in \ucircle \) and \( K \) is a convex set then \( F(K,u) \)
stands for the set of the boundary points
of \( K \) having outer normal \( u \). It is known that 
\begin{equation}\label{faces of diff body}
F(D K,u)=F(K,u)+F(-K,u)=F(K,u)-F(K,-u)
\end{equation}
(see \cite[Theorem 1.7.5(c)]{Schneider-1993}). If \( x \in\bd K \), then \( N(K,x) \),
the \emph{normal cone of \( K \) at \( x \)}, is defined as the set of all outer normal vectors to \( K \) at \( x \) together with \( o \). 

Given  \( K \in\cK_0^2\) and \( a,b\in\bd K \), let \( [a,b]_{\bd K} \) denote the set of points of \( \bd K \) which, in counterclockwise order, follow \( a \) and precede \( b \), together with \( a \) and \( b \). Let \( (a,b)_{\bd K} \) denote \( [a,b]_{\bd K}\setminus\{a,b\}\). We will refer to \( a \) as the \emph{left endpoint} of \( [a,b]_{\bd K} \) and to \( b \) as its \emph{right endpoint}.
Given an arc $\gamma$ on $\bd K$,  $\relint(\gamma)$  denotes $\gamma$ without its endpoints.

With \( K \in \cK^2 \) we also associate the \emph{support function \( h(K,\dotvar) \)} and the \emph{width function \( \wid(K, \dotvar) \)} defined for \( u \in \real^2 \) by
\begin{align*}
	h(K,u) & := \max_{x \in K} \sprod{u}{x},\\
	\wid(K,u) & := \max_{x \in K} \sprod{u}{x} - \min_{x \in K} \sprod{u}{x}.
\end{align*}
If \( K \in \cK_0^2\) and \( u \in \ucircle \), then \( \wid(K,u) \) is the Euclidean distance between the two distinct supporting lines of \( K \) orthogonal to \( u \).

For $K \in \cK_0^2$ and $o \in \intr(K)$ we introduce the \emph{radial function} $\rho(K,\dotvar)$ of $K$ by
\[
	\rho(K,u) := \max \setcond{\alpha \ge 0}{\alpha u \in K}.
\]
Geometrically, if $u \in \ucircle$, then $\rho(K,u)$ is the Euclidean distance from $o$ to the boundary point of $K$ lying on the ray emanating from $o$ and having direction $u$.

The \emph{mixed area} is the functional \( V : \cK^2 \times \cK^2 \rightarrow \real \) uniquely defined by the relation \( \vol(K+H) = \vol(K) + 2 V(K,H) + \vol(H) \) for all \( H,K \in \cK^2 \).

For a subset \( A \) of \( \real^2 \) the \emph{polar set} \( A^\circ \) of \( A \) is defined by
\[
	A^\circ : = \setcond{y \in \real^2}{\sprod{x}{y} \le 1 \ \forall x \in A}.
\]
It is well-known that the operation \( A \mapsto A^\circ \) is an involution on the set of all closed, convex sets that contain the origin.

\subsection{Norms and seminorms in \( \real^2 \), distributions}

\label{subsection seminorms}

We introduce seminorms using convex geometric notions as follows.
Let 
\[
	\cS^2 : = \setcond{B \subset \real^2}{B \text{ closed and convex,} \ B=-B,\ \intr B\neq\emptyset}.
\]
With \( B \in \cS^2 \) we associate the so-called Minkowski functional  \( \|\dotvar\|_B \) given by 
\begin{equation} \label{mink:funct:def}
        \|x\|_B := \inf \setcond{\alpha>0}{x \in \alpha B}.
\end{equation}
The functional \( \| \dotvar \|_B \) is a seminorm. Conversely, every seminorm in \( \real^2 \) can be expressed as \( \| \dotvar \|_B \) with an appropriate choice of \( B \in \cS^2 \). 
If \( \gamma \) is a rectifiable curve in \( \real^2 \), we can define \( \len_B (\gamma) \) to be the \emph{length of \( \gamma \) in the seminorm \( \| \dotvar\|_B \).} 
In analytic terms, \( \len_B (\gamma) \) can be expressed as the Stieltjes integral 
\(
	\len_B (\gamma) = \int_\gamma \| \dd x\|_B.
\)
Equivalently, if \( \gamma(s) \) is a parametrization of $\gamma$ in terms of Euclidean arc length, then  
\(
	\len_B (\gamma) = \int \| (d \gamma(s))/(ds)\|_B \dd s.
\)
We also let \( \len_B(\emptyset) := 0 \).

Using \( \len_B \) we define the \emph{perimeter-functional in the seminorm \( \| \dotvar\|_B \),} that is, the functional \( \perim_B : \cK^2 \rightarrow \real \) given by 
\begin{equation}
		\label{perim def}
        \perim_B (K)
			:= 
        \begin{cases}
                \len_B (\bd K)
					& \text{if} \ \intr K\neq\emptyset, 
				\\
                2 \len_B (K) & 
					\text{otherwise}.
        \end{cases}
\end{equation}
The functional $\perim_B$ is a valuation (see Subsection \ref{subsection valuations}).
In the following simple proposition we relate the geometry of $B$ with properties of $\perim_B$.

\begin{proposition}
	\label{prop on B and per B}
	Let $B \in \cS^2$. Then the following properties hold:
	\begin{enumerate}[(I)]
		\item 
			$\perim_B$ is identically equal to zero if and only if $B = \real^2$;
		\item 
			$B$ is unbounded (that is, $B$ is a strip or $B=\real^2$) if and only if there exist $\beta\geq0$ and $z\in \ucircle$ such that, for each $K \in \cK^2$, $\perim_B (K) = \beta w(K,z)$;
		\item 
			\label{B bounded} 
			$\perim_B$ is strictly positive on each $K \in \cK^2$ which is not a singleton if and only if $B$ is bounded;
		\item 
			$\perim_B$ is strictly monotone if and only if $B$ is strictly convex.
	\end{enumerate}
\end{proposition}

Assertions (I)--(III) of this proposition can be derived by straightforward methods; we omit the proofs. Regarding assertion \eqref{B bounded},
we observe that when \( B \in \cS^2 \) is bounded, \( \real^2 \) endowed with \( \|\dotvar\|_B \) becomes a two-dimensional normed space, sometimes also called a Minkowski plane. For related information on finite dimensional normed spaces see the survey \cite{Martini-Swanepoel-Weiss-2001} and the monograph \cite{Thompson-1996}. Assertion (IV) is a standard fact from the theory of Minkowski planes; see for example \cite[Proposition~2]{Martini-Swanepoel-Weiss-2001}.

We define the \emph{distribution \( \delta_{\gamma}^B \)} using Stieltjes integration by setting 
\[
        \pairing{\delta_{\gamma}^B}{\tau}
			:=
		\int_\gamma \tau(x) \, \|\dd x\|_B
			\qquad
		\forall \tau \in C^\infty(\real^2),
\]
where, as usual, $C^\infty(\real^2)$ denotes the space of functions on $\real^2$ differentiable infinitely many times. For information on the theory of distributions we refer to \cite{Hoermander-2003} and \cite{Gelfand-Shilov-1964}. By the Riesz representation theorem about positive linear functionals on the space of continuous functions \cite[\S2.2]{Rudin-1966}, the operation \(\tau \mapsto  \pairing{\Bdelta_{\gamma}}{\tau} \) is integration with respect to a nonnegative Borel measure on \( \real^2 \). Thus, we will interpret $\delta_{\gamma}^B$ either as a Borel measure or as a distribution.

When $B$ is the Euclidean ball \(\setcond{x \in \real^2}{|x| \le 1} \) rather than writing \( \len_B \), \( \perim_B \) and \( \delta_{\gamma}^B \) we merely write \( \len \), \( \perim \) and \( \delta_{\gamma} \).

\subsection{Monotone, translation invariant valuations on \( \cK^2 \)} 

\label{subsection valuations}

We shall deal with functionals \( \phi :\cK^2 \rightarrow \real \), which satisfy the following conditions:

\( \phi \) is a \emph{valuation}, i.e., $\phi(\emptyset)=0$ and
\begin{equation} \label{phi:is:val}
	\phi(K \cup H) = \phi(K) + \phi(H) - \phi(K \cap H) \quad \forall K, H \in \cK^2 \ \text{with} \ K \cup H  \in \cK^2;
\end{equation}

\( \phi \) is \emph{translation invariant}, i.e.,
\begin{equation} \label{phi:trans:inv}
	\phi(K+x) = \phi(K) \quad \forall K \in \cK^2 \text{ and } \forall x \in \real^2;
\end{equation}

\( \phi \) is \emph{monotone}, i.e., 
\begin{equation} \label{phi:monotone}
	\phi(K) \le \phi(H) \quad \forall K, H \in \cK^2 \ \text{with} \ K \subset H;
\end{equation}

	\( \phi \) is \emph{even}, i.e., 
\begin{equation} \label{phi:even}
	\phi(K) = \phi(-K) \quad \forall K \in \cK^2.
\end{equation}

There is no loss of generality in assuming that a valuation \( \phi \) on \( \cK^2 \) vanishes on singletons since this additional property can be ensured by replacing \( \phi \) with \( \phi - \phi(\{o\}) \). This change does not influence any of the above properties and it is possible to pass from $\gphi{K}$ to $g_{K,\phi-\phi(\{o\})}$, for each $K\in\cK^2$, via the formula $g_{K,\phi-\phi(\{o\})}=\gphi{K}-\phi(\{o\})$. Thus, we introduce  the family \( \Phi^2 \) as
\[
        \Phi^2 := \setcond{\phi}{\phi \ \text{satisfies \eqref{phi:is:val}--\eqref{phi:even} and} \ \phi(\{o\}) = 0}.
\]

It is well known that $\vol, \perim_B\in\Phi^2$.
Clearly, \( \vol \) is homogeneous of degree two while \( \perim_B \) is homogeneous of degree one, i.e., \( \vol(\lambda K) = |\lambda|^2 \vol(K) \)  and \( \perim_B( \lambda K) = |\lambda| \perim_B(K) \) for every \( \lambda \in \real \) and \( K \in \cK^2 \). 
It turns out that the above examples cover all important valuations belonging to \( \Phi^2 \). 
This is the content of the next theorem.

\begin{theorem} 
	\label{repr phi thm}
	Let \( \phi : \cK^2 \rightarrow \real\). Then the following conditions are equivalent:
	\begin{enumerate}[(i)]
		\item \( \phi \in \Phi^2\);
		\item there exist \( \alpha \ge 0 \) and an \( o \)-symmetric \( H \in \cK^2 \)  such that, for each $K \in \cK^2$,
		\begin{equation}
			\phi(K) = V(K,H) + \alpha \vol(K); \label{phi:as:mixed:area:and:vol}
		\end{equation}
		\item there exist \( \alpha \ge 0 \) and \( B \in \cS^2 \) such that, for each $K \in \cK^2$,
		\begin{equation} \label{phi:as:perim:and:vol}
			\phi(K) = \perim_B(K) + \alpha \vol(K).
		\end{equation}
	\end{enumerate}
	Furthermore, if (i),(ii) and (iii) are fulfilled, then the following statements hold:
	\begin{enumerate}[(I)]
		\item The parameter \( \alpha \ge 0 \) from (ii) and (iii) is uniquely determined by \( \phi \);
		\item The sets \( H \) and \( B \) from (ii) and (iii), respectively, are uniquely determined by \( \phi \) and are related to each other by the equalities
	\begin{align}
		\label{H B relation}
		H & = 2 \cR(B^\circ), 
		&
		B & = 2 \cR(H^\circ).
	\end{align}
	\end{enumerate}
\end{theorem}

This theorem follows rather directly from known results on valuations. Since we have not found any source explicitly containing it, we present a  proof.

\begin{proof}[Proof of Theorem~\ref{repr phi thm}]

	\emph{(i) \( \Rightarrow \) (ii).} Let \( \phi \in \Phi^2\). It is known that every monotone, translation invariant valuation on \( \cK^n \) is continuous (see \cite[Theorem~8]{McMullen-1977}) and that every continuous translation invariant valuation on \( \cK^n \) is a sum of \( n+1 \) continuous, translation invariant valuations which are positively homogeneous of degree $i$, for $i = 0,\ldots,n$ (see  \cite[p.~38]{McMullen-1990} and \cite[Theorem~9]{McMullen-1977}). Thus \( \phi=\phi_1 + \phi_2 \), where \( \phi_1 \) is homogeneous of degree one and \( \phi_2 \) is homogeneous of degree two. It is not hard to see that \( \phi_1 \) and \( \phi_2 \) are determined by \( \phi \) as follows:
	\begin{align}
		\phi_1(K) =& \lim_{\lambda \rightarrow +0} \frac{\phi(\lambda K)}{\lambda}, \label{phi:1:repr}\\
		\phi_2(K) =& \lim_{\lambda \rightarrow +\infty} \frac{\phi(\lambda K)}{\lambda^2}. \label{phi:2:repr}
	\end{align}
	
	Since $\phi \in \Phi^2$, the above expressions for $\phi_1$ and $\phi_2$ imply $\phi_1, \phi_2 \in \Phi^2$.  It is known that every continuous translation invariant valuation on \( \cK^n \), which is homogeneous of degree \( n \) coincides with the volume, up to a constant multiple (see \cite[2.1.3]{MR21:1561}). Thus, \( \phi_2 = \alpha \vol \) for some \( \alpha \in \real \). The value \( \alpha \) is nonnegative since otherwise \( \phi_2 \) would not be monotone in the sense of \eqref{phi:monotone}. Monotone translation invariant valuations on \( \cK^n \) of degree \( 1 \) and \( n-1 \) have been characterized in terms of mixed volumes in \cite[Theorem~1]{McMullen-1990} and \cite{MR0640871}, respectively. Each of these characterizations implies that \( \phi_1(\dotvar) = V(\dotvar,H) \) for some \( H \in \cK^2 \). Using the evenness of $\phi_1$ and standard properties of mixed area we see that, in the representation of $\phi_1$ in terms of $H$, the set $H$ can be replaced by $\frac{1}{2} DH$. Thus, we can assume that $H$ is $o$-symmetric.
	
	\emph{(ii) \( \Rightarrow \) (i)} follows from standard properties of mixed volumes.

	\emph{(ii) $\Leftrightarrow$ (iii)}. 
		It is known and easy to see that the operation $B \mapsto H = \cR(B^\circ)$ is a bijection on the set $\cS^2 \cap \cK_0^2$. From basic properties of the polarity, we also conclude that the above operation is an involution on $\cS^2 \cap \cK_0^2$, meaning $B=\cR(H^\circ)$. Furthermore, we observe that the above operation maps bijectively the set of $o$-symmetric strips $B$ to the set of $o$-symmetric segments $H$, and in the latter (degenerate) situation the inversion formula $H = \cR(B^\circ)$ still remains valid. 

		In view of the above observations, in order to conclude the proof of the equivalence (ii) $\Leftrightarrow$ (iii) it suffices to show  $\perim_B (K) = 2 V(K, \cR(B^\circ))$ for every $K \in \cK^2_0$ and $B \in \cS^2$. In the case $B \in \cS^2 \cap \cK_0^2$ this is known, see \cite[Equalities (4.8) at p.120]{Thompson-1996}.
		When $B$ is $\real^2$ or an an $o$-symmetric strip the equality can be verified in a straightforward manner.

	Assertion (I) holds because $\phi_2$ is determined by $\phi$ via \eqref{phi:2:repr} and $\alpha = \phi_2([0,1]^2)$. For proving (II) we observe that (i) and (ii) imply $V(K,2 \cR(B^\circ)) = V(K,H)$ for every $K \in \cK^2$. It is well-known and not hard to show that a nonempty, $o$-symmetric set $H \in \cK^2$ is determined by the knowledge of $V(K,H)$ for every $K \in \cK^2$ (in fact, it suffices to know $V(K,H)$ for every $o$-symmetric segment $K$). Thus $2 \cR(B^\circ) = H$. 
\end{proof}

\section{
	Representation of  $\phi$-covariograms in terms of convolutions
}
	\label{sect:covariograms_as_convolutions}
  In the following theorem we present a functional-analytic expression for \( \gphi{K} \).

\begin{theorem} 
	\label{g phi thm}
	Let \( \phi \in \Phi^2 \setminus \{0\} \) and \( K \in \cK^2_0 \). Then the following assertions hold:
	\begin{enumerate}[(I)]
		\item 
		Almost everywhere on \( \real^2 \), in the sense of Lebesgue measure, we have
		\begin{equation}\label{phicov as convolution}\begin{split}
			\gphi{K}
				& = 
				1_K \ast \Bdelta_{- \bd K} + \Bdelta_{\bd K} \ast \chf_{-K} + \alpha \chf_K \ast \chf_{-K} 
				\\
				& = 
				\left( \Bdelta_{-\bd K} + \frac{\alpha}{2} \chf_{-K} \right) \ast \chf_K + \left( \Bdelta_{\bd K} + \frac{\alpha}{2} \chf_K \right)  \ast \chf_{-K}.
        \end{split}\end{equation}
        \item
		\label{integral of g phi}
			\(
				\int_{\real^2} \gphi{K}(x) \dd x = \vol(K) (2 \perim_B(K) + \alpha \vol(K)).
			\)
		\item \label{g phi thm_a}
		\( \supp \gphi{K} = DK \).
		\item $\gperim{K}$ and $\sqrt{g_K}$ are concave on $DK$.
	\end{enumerate}
\end{theorem}
\begin{proof}
	In view of \eqref{decomposition_of_phi-cov}, the assertion for a general $\phi \in \Phi^2$ follows by proving the assertion when $\phi =\perim_B$, with $B \in \cS^2$, and when  $\phi$ is the volume. When $\phi=\vol$, assertions (I)--(IV) are known. In this particular case (I) and (II) can be found in \cite[p.85, (4.3.1) and (4.3.2)]{MR0385969}, (III) is trivial and well known, while the proof of the concavity of $\sqrt{g_K}$ in the assertion (IV) can be found in \cite[Proof of Theorem~7.3.1]{Schneider-1993}.
	Consider the case $\phi=\perim_B$. 
		
	For showing (I) it suffices to verify that almost everywhere, in the sense of Lebesgue measure on \( \real^2 \), we have
	\begin{gather}
		\gperim{K}(x)
				 =
		\len_B(K \cap (\bd K + x)) + \len_B(K \cap (\bd K -x)),
				\label{g perim via chords}
		\\
		\intertext{and}
		\len_B(K \cap (\bd K + x)) 
				 = 
		(\chf_K \ast \delta_{-\bd K}^B)(x), 
				\label{chord len as convol} 
	\end{gather}
	Equality \eqref{g perim via chords} obviously holds for \( x \in \real^2 \setminus DK \), since in this case $K \cap (K + x)=\emptyset$  and both the left and the right hand side are zero. Let
	\[
		A := \intr (DK) \setminus \bigcup_{u \in \ucircle} \bigl(F(K,u) - F(K,u) \bigr).
	\]
	There are at most countably many directions $u \in \ucircle$ for which $F(K,u)$ is one-dimensional. For those directions $F(K,u)-F(K,u)$ is one-dimensional as well. For all the remaining directions $u$, one has $F(K,u) =F(K,u)-F(K,u)= \{o\}$. Thus, the union for $u \in \ucircle$ in the definition of $A$ has volume zero and, as a consequence,  \( \vol(A)=\vol(DK) \). Observe that, for every $x \in A$, $\bd K\cap (\bd K+x)$ consists of two points, the convex body $K$ has precisely two chords which are translates of $[o,x]$. and, moreover, the relative interior of both these chords is contained in $\intr K$. The latter implies that \eqref{g perim via chords} holds for every \( x \in A \). Hence \eqref{g perim via chords} holds almost everywhere.

	Let us show \eqref{chord len as convol}. Consider an arbitrary \( \tau \in C^\infty(\real^2) \). Using the definition of convolution of distributions (see \cite[Chapter~I, \S5]{Gelfand-Shilov-1964}) and performing changes of variable of integration, we obtain
	\begin{align}
			\nonumber
		\pairing{ \chf_K \ast \Bdelta_{- \bd K} }{ \tau }
			= 
		\int_{-\bd K} &  \left\{ \int_{\real^2} \chf_K(x) \tau(x+y) \dd x \right\} \| \dd y\|_B 
			\\ 
			\nonumber
			= 
		\int_{\phantom{-}\bd K} & \left\{ \int_{\real^2} \chf_K(x) \tau(x-y) \dd x \right\} \|\dd y\|_B 
			\\
			\label{aux integral}
			= 
		\int_{\phantom{-}\bd K} & \left\{ \int_{\real^2} \chf_K(x+y) \tau(x) \dd x \right\} \|\dd y\|_B.
	\end{align}
	We recall that the Stieltjes integration on $\bd K$ can be expressed as integration with respect to a Borel measure, which we denote by $\delta_{\bd K}^B$. Thus, $\vol \times \delta_{\bd K}^B$ is a product of two Borel measures and, by this, again a Borel measure. The function $\chf_K(x+y) \tau(x)$ on $\real^2 \times \real^2$ is clearly Borel measurable and, moreover, summable with respect to $\vol \times \delta_{\bd K}^B$. By Fubini's theorem  \cite[Theorem~8.8]{Rudin-1966} we can exchange the order of integration in \eqref{aux integral} arriving at
	\begin{align*}
		\pairing{ \chf_K \ast \Bdelta_{-\bd K} }{ \tau }
			& = 
		\int_{\real^2} \left\{ \int_{\bd K} \chf_K(x+y) \|\dd y\|_B \right\} \tau(x) \dd x 
			\\
			& = 
		\int_{\real^2} \len_B(\bd K \cap (K-x)) \tau(x) \dd x 
			\\
			& = 
		\int_{\real^2} \len_B(K \cap (\bd K+x)) \tau(x) \dd x.
	\end{align*}
	Hence we get \eqref{chord len as convol}. This concludes the proof of (I).

	Assertion (II) is a direct consequence of (I). Assertion (III) follows from the fact that \( \intr(K \cap (K+x))\neq\emptyset \) for every \( x \in \intr DK \). This implies, by Proposition~\ref{prop on B and per B}, that $\gperim{K}(x)$ is positive for every $x \in \intr DK$. 

	It remains to verify (IV). 
	Consider \( x,y \in DK \) and \( 0 \le \lambda \le 1 \). The inclusion
	\begin{equation}
		\label{incl self-intersector}
		(1-\lambda) (K \cap (K+x)) + \lambda (K \cap (K+y)) 
			\subset
		K \cap (K + (1-\lambda) x+ \lambda y)
	\end{equation}
	can be verified in a straightforward manner. Representing $\perim_B$ in terms of mixed areas according to Theorem~\ref{repr phi thm} and using the monotonicity and the linearity of the mixed areas (in any of the two arguments) we get
	\( 
		\gperim{K}((1-\lambda) x + \lambda y) 
			\ge
		(1-\lambda) \gperim{K}(x) + \lambda \gperim{K}(y).
	\)
\end{proof}

\section{Radial derivatives of \( \phi \)-covariograms} 

\label{sect deriv}

One of the tools in the proofs of the retrieval results will be the formulas which provide a geometric interpretation of the radial derivatives of \( \gperim{K} \) and \( g_{K} \). We introduce some notations illustrated by Fig.~\ref{inscribed P fig}. Fix $K \in \cK_0^2$ and $x \in \intr (DK) \setminus \{o\}$. We introduce a number of objects that depend on the pair $(K,x)$ but for the sake of brevity we mostly only indicate the dependence on $x$. Let \( \inspar(x) \) be a parallelogram inscribed in \( K \) (which means, that all vertices of \( \inspar(x) \) belong to \( \bd K \)) and such that two opposite edges of \( \inspar(x) \) are translates of the segment \( [o,x] \).  The parallelogram $\inspar(x)$ is determined uniquely unless \( K \) has a one-dimensional face parallel to \( x \) and strictly longer than \( [o,x] \). In the case of non-uniqueness we just fix any \( \inspar(x) \) satisfying the above conditions. Furthermore, for every \( x \in \intr DK \setminus \{o\} \)  we choose \( \inspar(x) \) and \( \inspar(-x) \) to be equal. Let \( p_1(x),\ldots, p_4(x) \) be the vertices of \( \inspar(x) \) in counterclockwise order on \( \bd K \) and such that \( x=p_1(x)-p_2(x)=p_4(x)-p_3(x) \).

\begin{FigTab}{c}{1.4mm}
       \begin{picture}(54,50)
		\put(-1,-0.7){
			\includegraphics[width=50\unitlength]{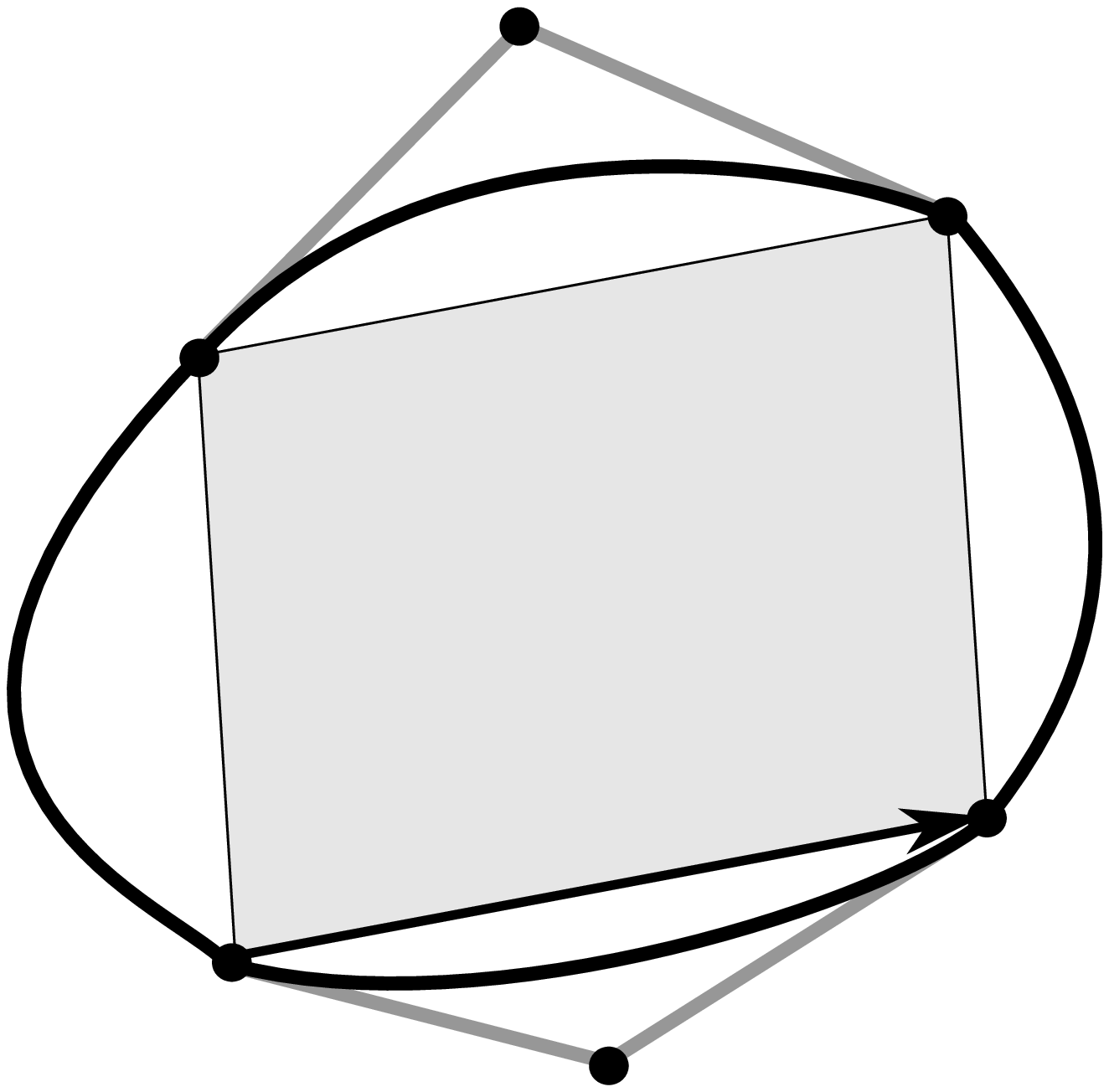}
		}
		\put(44,40){\small \( p_1(x) \)}
		\put(3.5,34.5){\small \( p_2(x) \)}
		\put(7,2){\small \( p_3(x) \)}
		\put(43,8){\small \( p_4(x) \)}
		\put(20,42){\scriptsize \( \arc(x) \)}
		\put(22,9){\small \( x \)}
		\put(15,47){\small \( p_{1,2}(x) \)}
		\put(24,-3){\small \( p_{3,4}(x) \)}
		\put(30,45){\scriptsize \( \thecap(x) \)}
       \end{picture}
   \\
   \parbox[t]{0.9\textwidth}{\caption{\label{inscribed P fig}} Data associated to \( K \) and \( x \in \intr DK \setminus \{o\} \): the points \( p_1(x),\ldots,p_4(x) \), \( p_{1,2}(x), p_{3,4}(x) \), the parallelogram \( \inspar(x) \) inscribed in \( K \) (shaded) and the boundary arc \( \arc(x) \) joining \( p_1(x) \) and \( p_2(x) \)}
   \end{FigTab}

It is known \cite{Matheron8601} that for $u \in \ucircle$ and $0 < s < \rho(DK,u)$, the value $- \pderiv{s} g_K(su)$ is the Euclidean distance between the lines $\aff \{p_1(su),p_2(s u)\}$ and $\aff \{p_3(s u), p_4 (s u)\}$. This  can be expressed in the following equivalent way.

\begin{theorem}
	\label{radial deriv of g}
	\thmheader{On radial derivative of the standard covariogram \cite{Matheron8601}.}
	Let \( K \in \cK^2_0 \) and let \( x \in \intr DK \setminus \{o\} \). Then
	\begin{equation}
		- \PartialDerivAt{ g_{K}(t x) }{t}{1} 
			= 
		\vol\bigl(\inspar(K,x)\bigr).
			\label{rad deriv of g}
	\end{equation}
\end{theorem}

We observe that, in contrast to $\frac{\partial}{\partial t}g_{K} (t x)$, the derivative
\(
	\frac{\partial}{\partial t}\gperim{K} (t x)
\)
does not always exist in the classical sense. Nevertheless, both the left and the right derivatives do exist, as a consequence of  the concavity of \( \gperim{K} \) on $DK$. Theorem~\ref{radial deriv of g perim} below presents a  geometric interpretation of the left derivative.

Given  $K\in\cK^2_0$ and $p\in\bd K$  we denote by \emph{left tangent} (and by \emph{right tangent}) of $K$ at $p$ the line tangent at $p$ to the portion of $\bd K$ which precedes $p$ (which follows $p$, respectively).

Let $x\in\intr DK\setminus\{o\}$, \( l_1(x) \) be the right tangent of \( K \) at \( p_1(x) \) and \( l_2(x) \) be the left tangent of \( K \) at \( p_2(x) \). Define
\[
	\arc(x)
		:=
	\bigl[p_1(x),p_2(x) \bigr]_{\bd K}.
\]
Assume  \( \arc(x) \ne [p_1(x),p_2(x)] \). In this case  \( l_1(x) \) and \( l_2(x) \) are not parallel to \( [p_1(x),p_2(x)] \).  These lines are also not parallel to each other, because this may happen only if they are lines supporting $K$ on opposite sides and this cannot be due to the assumption $x\in\intr DK$. We denote by \( p_{1,2}(x) \) the intersection point of \( l_1(x) \) and \( l_2(x) \). When \( \arc(x) = [p_1(x),p_2(x)] \), then both \( l_1(x) \) and \( l_2(x) \) are parallel to \( [p_1(x),p_2(x)] \) and we denote by \( p_{1,2}(x) \) any point on \( [p_1(x),p_2(x)] \). We introduce the polygonal line
\[
	\thecap(x)
		:=
	\bigl[p_1(x),p_{1,2}(x)\bigr] \cup \bigl[p_{1,2}(x),p_2(x)\bigr].
\]

Similarly, let \( l_3(x) \) be the right tangent of \( K \) at \( p_3(x) \) and \( l_4(x) \) be the left tangent of \( K \) at \( p_4(x) \). If \( [p_3(x),p_4(x)]_{\bd K} \ne [p_3(x),p_4(x)] \), then we denote by \( p_{3,4}(x) \) the intersection point of \( l_3(x) \) and \( l_4(x) \), otherwise \( p_{3,4}(x) \) is chosen to be any point on \( [p_3(x),p_4(x)] \). Clearly, one has
\[
	\thecap(-x)
		= 
	\bigl[p_3(x),p_{3,4}(x)\bigr] \cup \bigl[p_{3,4}(x),p_4(x)\bigr].
\]

\begin{theorem}
	\thmheader{On radial derivatives of the perimeter-covariogram.}
	\label{radial deriv of g perim}
	Let \( K \in \cK^2_0 \) and \( x \in \intr DK \). Then
	\begin{equation} \label{rad:deriv:perim:cov:repr}
		- \left. \pderivleft{t} \gperim{K}(tx) \right|_{t=1}
			= 
		\len_B \left(\thecap(K,x)\right) + \len_B \left(\thecap(K,-x)\right).
	\end{equation}
\end{theorem}

In order to prove Theorem~\ref{radial deriv of g perim} we need to introduce some notation and prove a preliminary lemma.
For a convex function \( f \) defined on an interval in \( \real \) the right derivative of \( f \)  will be denoted by  \( \rderiv{f} \).
\begin{lemma} 
	\label{piece of tangent lem}
	Let \( B \in \cS^2 \). Let \( f : [0,1] \rightarrow \real \) be a  convex function such that \( f(0)=0 \) and \( \rderiv{f}(0) \ge 0 \). For every \( 0 < s \le 1 \) we define
	\begin{align*}
		b(s) &:= \len_B \left(\setcond{(x,f(x))}{0 \le x \le s}\right) , 
		\\
		b^+(s) &:= \len_B \left(\setcond{(x,\rderiv{f}(0) x)}{0 \le x \le s}\right), 
	\end{align*}
	Then, as \( s \rightarrow +0 \), one has $b(s) - b^+(s)=o(s)$.
\end{lemma}
\begin{proof}
	All asymptotic expansions in this proof are considered for \( s \rightarrow +0 \). Taking into account \( f(0)=0 \) and using the definition of \( \rderiv{f} \) we obtain  
	\begin{equation}
		\label{f(s) asympt}
		f(s) = \rderiv{f}(0) s + o(s).
	\end{equation}
	Hence
	\begin{equation*}
		\delta(s) :=f(s) - s \rderiv{f}(0) = o(s).	
	\end{equation*}
	We introduce
	\[
		b^-(s) = \len_B \left(\setcond{\left(x,\frac{f(s)}{s} x\right)}{0 \le x \le s}\right),
	\] 
	\( p(s)=(s, \rderiv{f}(0) s) \) and \( q(s)=(s,f(s)) \).  Observe that
	\[
		\len_B ([p(s),q(s)]) = \delta(s) \| (0,1)\|_B.
	\]
	We recall that $\perim_B$ is a monotone valuation, by Theorem~\ref{repr phi thm}. The inclusions
	\[
		[o,q(s)]
			\subset
		\conv\left(\{o,q(s)\}\cup\setcond{(x,f(x))}{0 \le x \le s}\right)
			\subset
		\conv\{o,p(s),q(s)\}
	\]
	together with the definition of $\perim_B$ (see \eqref{perim def}) imply
	\[
	 b^-(s)\leq b(s)\leq b^+(s) + \delta(s) \|(0,1)\|_B.
	\]
	The inclusion $[o,p(s)] \subset \conv\{o,p(s),q(s)\}$ and the definition of $\perim_B$ imply
	\[
	 b^+(s)-\delta(s) \| (0,1)\|_B\le b^-(s).
	\]
	(The latter is just a triangle inequality for points $o,p(s),q(s)$ with respect to the seminorm $\|\dotvar\|_B$.) Consequently,
	\(
		|b(s)- b^+(s)| \le \delta(s) \|(0,1)\|_B= o(s),
	\)
	which yields the assertion.
\end{proof}

\begin{proof}[Proof of Theorem~\ref{radial deriv of g perim}]
	Let $x \in \intr DK\setminus \{o\}$. Since \( \inspar(x) = \inspar(-x) \) we have 
	\begin{align*}
        \gperim{K}(x) = \perim_B(K) - \len_B (\arc(x)) - \len_B (\arc(-x)).
	\end{align*}
	It suffices to show that the left derivative
	\[
		a(x):=\left. \pderivleft{t} \len_B (\arc(t x)) \right|_{t=1}
	\]
	exists and is equal to $\len_B (\thecap(x))$. In the case \( \arc(x) = [p_1(x),p_2(x)] \) it is easy to verify that \( a(x) = \|x\|_B = \len_B (\thecap(x)) \).  Assume that \( \arc(x) \ne [p_1(x),p_2(x)] \).  Then \( l_1(x) \) and \( l_2(x) \) are both not parallel to \( x \).  Changing a coordinate system in \( \real^2 \) with an appropriate nonsingular affine transformation, without loss of generality we can assume that \( x=(0,1) \) and \( \inspar(x)=[0,1]^2 \). Then we can introduce an \( \eps>0 \) and convex functions \( f_1, f_2 : [0,\eps] \rightarrow \real \) with \( f_1(0)=f_2(0)=0 \) such that 
	\begin{align*}
		\setcond{(-s,f_1(s))}{0 \le s \le \eps}
			\subset & \bd K, 
		\\
		\setcond{(-s,1-f_2(s))}{0 \le s \le \eps}
			\subset & \bd K.
	\end{align*}
	For every sufficiently small \( t \ge 0 \)  one can uniquely define the parameter \( s(t) \ge 0 \) such that \( [p_1((1-t)x),p_2((1-t) x)] \subset \{-s(t)\} \times \real \).  In other words, \( s(t) \) is the distance between \( \aff [p_1(x),p_2(x)] \) and \( \aff [p_1((1-t)x),p_2((1-t)x)] \).  For \( i \in \{1,2\} \) let us define \( b_i(s), b^+_i(s) \) with respect to the function \( f_i(s) \) in the same way as \( b(s), b^+(s) \) are defined in Lemma~\ref{piece of tangent lem} with respect to a function \( f(s) \).  Let also $\delta_i(s) := f_i(s) - \rderiv{f_i}(0) s$ for $i \in \{1,2\}$. The function \( a(x) \) can be expressed as
	\[
        a(x) := 
		\lim_{t \rightarrow +0} 
		\frac{1}{t} \Bigl(\len_B (\arc(x))- \len_B (\arc((1-t)x)) \Bigr).
	\]
	In the rest of the proof we shall consider asymptotic behaviors  for \( t \rightarrow +0 \). Note that \( s(t) \rightarrow +0 \) as \( t \rightarrow +0 \). Let us determine the asymptotic behavior of 
	\[
        a_t(x) := 
		\frac{1}{t} \Bigl(\len_B (\arc(x))- \len_B(\arc((1-t)x)) \Bigr).
	\]
	To this end we shall use Lemma~\ref{piece of tangent lem} and the relation 	
	\begin{equation}
		\label{t=f1+f2}
		t=f_1(s(t))+f_2(s(t)),  
	\end{equation}
	which holds by construction.  

	In the following computations, for the sake of brevity we write \( f_i \) rather than \( f_i(s(t)) \). Analogously, we also omit the explicit indication of the dependency on \( s(t) \) for \( \delta_i(s(t)) \), \( b_i(s(t)) \) and \( b_i^+(s(t)) \) (where \( i \in \{1,2\} \)). 

	We shall determine the limit of 
	\begin{align*}
        a_t(x)
			=& \frac{1}{t}(b_1+b_2) = \frac{1}{t} (b_1^+ + b_2^+) + \frac{1}{t}(b_1 - b_1^+ + b_2 - b_2^+),
	\end{align*}
	as $t \rightarrow +0$.
	In view of \eqref{t=f1+f2} and Lemma~\ref{piece of tangent lem} one has
	\begin{align}
		\label{rest term}
		\frac{1}{t}(b_1 - b_1^+ + b_2 - b_2^+) = \frac{o(s(t))}{f_1 + f_2} = \frac{o(s(t))}{c \cdot  s(t) + o(s(t))},
	\end{align}
	where 
	\[
		c = \rderiv{(f_1+f_2)}(0).
	\]

	Note that $c>0$. This can be shown arguing by contradiction. Assume that $\rderiv{(f_1+f_2)}(0) = 0$. Then $\rderiv{f_1}(0)= \rderiv{f_2}(0) = 0$. It follows that the body $K$ has parallel supporting lines at points  $p_1(x)$ and $p_2(x)$. The latter yields $x \in \bd DK$, contradicting the assumption $x \in \intr DK \setminus \{o\}$. Taking into account $c > 0$, we conclude that the term \eqref{rest term} converges to $0$, as $t \rightarrow +0$. Thus, it remains to determine the limit of $\frac{1}{t} (b_1^+ + b_2^+)$.

	Taking into account \eqref{t=f1+f2}, we obtain

	\begin{align*}
        \frac{1}{t} (b_1^+ + b_2^+) 
			& = 
		\frac{b_1^+ + b_2^+}{t-\delta_1 - \delta_2} \cdot \frac{t-\delta_1 - \delta_2}{t} 
			\\ & = 
		\frac{b_1^+ + b_2^+}{t-\delta_1 - \delta_2} \cdot \frac{f_1 + f_2 -\delta_1 - \delta_2}{f_1 + f_2} 
			\\ & = 
		\frac{b_1^+ + b_2^+}{t-\delta_1 - \delta_2} \cdot \frac{c \cdot s(t)}{c \cdot s(t) + o(s(t))}
	\end{align*}
	The quotient
	\[
		\frac{c \cdot s(t)}{c \cdot s(t) + o(s(t))}
	\]
	goes to $1$, as $t \rightarrow +0$. Let us analyze the other quotient 
	\[
		\frac{b_1^+ + b_2^+}{t-\delta_1 - \delta_2}.
	\]
	Consider the triangle \( T:=\conv \{p_1(x),p_{1,2}(x),p_2(x)\} \). For the sake of brevity we shall write $p_1, p_2, p_{1,2}$ omitting the explicit dependence on $x$.  The section \( T \cap (\{-s(t)\} \times \real) \) of $T$ has Euclidean length \( 1-t + \delta_1 + \delta_2 \). 
	We introduce points \( p_1^+ \) and \( p_2^+ \) such that \( [p_1^+,p_2^+]=T \cap (\{-s(t)\} \times \real) \) and \( p_i^+ \in [p_{1,2},p_i] \) for \( i \in \{1,2\} \).
	The edge \( [p_1,p_2] \) of \( T \) has Euclidean length one. 
	Thus, using the homothety of \( T \) and \( \conv \{p_1^+,p_2^+,p_{1,2}\} \), we get for 
	\[
        \frac{\|p_i - p_{1,2}\|_B}{1} = \frac{\|p_i - p_{1,2}\|_B - b_i^+}{1 -t + \delta_1 + \delta_2} \qquad \forall i \in \{1,2\}.
	\]
	The latter amounts to 
	\[
        (t - \delta_1 - \delta_2) \|p_i - p_{1,2}\|_B = b_i^+ \qquad \forall i \in \{1,2\}.
	\]
	Hence
	\[
		\frac{b_1^+ + b_2^+}{t-\delta_1 - \delta_2} = \|p_1 + p_{1,2}\|_B + \|p_2 + p_{1,2}\|_B.
	\]
	Summarizing we conclude that $a_t(x)$ goes to $\|p_1 + p_{1,2}\|_B + \|p_2 + p_{1,2}\|_B$, as $t \rightarrow +0$.
\end{proof}


\section{Retrieval results}  
\label{sect retrieval}

The proof of Theorem~\ref{retrieval centr sym thm} follows closely that of the corresponding result for $g_K$. It is based on three ingredients. The first one is Brunn-Minkowski inequality and the characterization of its equality cases. The second one is Theorem~\ref{g phi thm} (Assertions~(II) and (III)). The third one, not present in the case of $g_K$, is the linearity of $\perim_B$ with respect to Minkowski addition.

The proof of Theorem~\ref{retrieval of polygons thm} has the same structure of that of the determination of a convex polygon $P$ by $g_P$ contained in \cite{MR1909913}.
It is roughly divided in two steps.
In the first step (Lemma~\ref{curv info lem}) one uses the shape of $\supp \gphi{P}$ and the asymptotic behavior of $\gphi{P}$ near $\bd\supp \gphi{P}$ to determine some information on $\bd P$. This information is only local and determined up to a reflection of $P$.
For instance for each $u\in\ucircle$ one can determine whether the two lines orthogonal to $u$ and supporting $P$ intersect $\bd P$ in a vertex and an edge or in two vertices or in two edges, and one can determine the length of these edges and the normal cone at these vertices. However this is known up to a reflection of $P$, and thus at this stage we do not know, for instance, which of the two supporting lines contains an edge and which a vertex.
If $Q$ denotes a polygon with $\gphi{P}=\gphi{Q}$, this leads naturally to a decomposition of $\bd P$ in a finite number of pairs of antipodal arcs with the property that each pair of arcs is also contained in a suitable translation or reflection of $\bd Q$, with these translations and reflections that a priori may vary from pair to pair.
It is the goal of the second step to prove that they are the same for all pairs.
This is done via Lemma~\ref{symmetric arcs lem}, which proves that every pair of \emph{maximal} antipodal arcs contained in $\bd P\cap \bd Q$ consists of two  arcs which are  reflections of each other. This proves that ``the reflection does not matter'' and opens the way to the conclusion.
One key ingredient in the second step is the geometric interpretation of the radial derivative of $\gperim{P}$ provided by Theorem~\ref{radial deriv of g perim}. 

The proof of Theorem~\ref{retrieval for width-covar thm} is still structured in  the same two steps. However each step has to be proved following new ideas.  In the first step (Lemmas~\ref{lem:description_of_core} and~\ref{lem:curv_info_width}) we use the possibility of identifying a certain subset of $\supp \gphi{P}$, which we call $\core{P}$ (it is the subset consisting of $x\in\supp \gphi{P}$ such that $\gphi{P}(x)=\wid(P,z)-\sprod{x}{z}$), and to read in $\core P$ some information about $P$.
Regarding the second step,  the key lemma holds in a weaker form when $\phi(\cdot)=\wid(\cdot,z)$. Indeed  the proof of Lemma~\ref{symmetric arcs lem} rests ultimately on the fact that there is a strict inequality between the values of \( \phi \) on two triangles (i.e. the triangles \( \conv\{c_1,c_2,c_3\} \) and \( \conv\{d_1,d_2,d_3\} \) in Fig.~\ref{fig:arcs}) because one is strictly contained in a translation of the other.
Since the width is not strictly monotone, a strict inequality holds only under some assumptions on the position of the triangles with respect to \( z \). The weak form of this lemma, contained in Lemmas~\ref{lem:w_symmetric_arcs} and~\ref{lem:w_symmetric_arcs_conclusion}, is still sufficient to conclude.

\subsection{Retrieval result for centrally symmetric convex bodies (Theorem~\ref{retrieval centr sym thm})}

\begin{proof}[Proof of Theorem~\ref{retrieval centr sym thm}] Let  $H\in\cK^2_0$ be such that $\gphi{K}=\gphi{H}$. Theorem~\ref{g phi thm} implies
\begin{gather}
 DK=DH\label{censymm1},\\
 2\vol(K)\perim_B(K)+\alpha\left(\vol(K)\right)^2=2\vol(H)\perim_B(H)+\alpha\left(\vol(H)\right)^2. \label{censymm2}
\end{gather}
Equality~\eqref{censymm1}, the possibility of representing $\perim_B$ as a mixed area and the linearity  of the mixed area imply
\begin{equation}\label{censymm3}
 \perim_B(K)=\frac1{2}\perim_B(DK)=\perim_B(H).
\end{equation}
Equality~\eqref{censymm1} and the Brunn-Minkowski inequality (see \cite[Theorem~7.3.1]{Schneider-1993}) imply
\begin{equation}\label{censymm4}
\vol(H)\leq \vol(K),
\end{equation}
with equality if and only if $H$ is centrally symmetric . Formulas~\eqref{censymm2}, \eqref{censymm3} and~\eqref{censymm4} imply $\vol(H)=\vol(K)$ and, as consequence, the central symmetry of $H$. Note that a centrally symmetric convex body coincides, up to translation,  with its difference body scaled by $1/2$, that is, with the support of its $\phi$-covariogram scaled by $1/2$.
\end{proof}

%
%
%
%

\subsection{
	Determination of polygons from covariograms generated by strictly monotone valuations
	(Theorem~\ref{retrieval of polygons thm}).
}

\label{section strictly monotone}

Following Bianchi \cite{MR1909913}, given \( u\in \ucircle \), the \emph{curvature information} \( \ci(P,u) \) of a convex polygon \( P \subset \real^2 \) at $u$  is defined by
\begin{equation*}
	\ci(P,u):=\begin{cases}\len (F(P,u))& \text{if $F(P,u)$ is an edge,}\\
	           N(P,a)& \text{if $F(P,u)=\{a\}$ for some vertex $a$ of $P$.}
	          \end{cases}
\end{equation*}
More informally, \( \ci(P,u) \) provides the knowledge of whether \( F(P,u) \) is an edge or a vertex together with the length of \( F(P,u) \), when \( F(P,u) \) is and edge, and with the normal cone of \( P \) at \( F(P,u) \), when \( F(P,u) \) is a vertex.

\begin{lemma}
	\label{curv info lem}
	Let \( \phi\in \Phi^2 \setminus \{0\} \) be  strictly monotone. Let \( P \) be a convex polygon in $\real^2$ and $u \in \ucircle$. Then \( g_{P,\phi} \) determines the set
	\begin{equation*}
		\{\ci(P,u), \ci(-P,u)\}.
	\end{equation*}
\end{lemma}

\begin{remark}
 The concept of \emph{synisothetic} pairs of convex sets has been introduced and used in \cite{BiCrosscov09} and \cite{Bianchi-2009-polytopes}.  We remark that the conclusion of Lemma~\ref{curv info lem} can be expressed in terms of synisothesis as follows. If \( P \) and \( Q \) are convex polygons with \( g_{P,\phi}=g_{Q,\phi} \) then  \( (P,-P) \) and \( (Q,-Q) \) are synisothetic. 
\end{remark}

\begin{proof}[Proof of Lemma~\ref{curv info lem}]

The proof of this lemma is divided into the proofs of Claims~\ref{opposite lengths claim}, \ref{determ vertex-edge claim} and \ref{determ vertex-vertex claim}. 
We recall that $DP=\supp\gphi{P}$ and that  we assume that the $\phi$-covariogram decomposes as  in \eqref{decomposition_of_phi-cov}.

\begin{claim}
	\label{opposite lengths claim}
	The function \( g_{P,\phi} \) determines \( \{\len F(P,u),\len F(P,-u)\}. \)
\end{claim}

\begin{proof}
	If \( F(DP,u) \) is a vertex, then both \( F(P,u) \) and \( F(P,-u) \) are vertices, by \eqref{faces of diff body}. Assume that \( F(DP,u) \) is an edge. The knowledge of \( DP \) gives
	\begin{equation}\label{sum_lengths_opposite_edges}
		\len (F(DP,u)) = \len (F(P,u))+\len (F(P,-u)),
	\end{equation}
	due to \eqref{faces of diff body}. Let \( x_0 \) be the midpoint of \( F(DP,u) \). One has
	\[
		g_{P,\phi}(x_0)=\min\{\len_B (F(P,u)),\len_B (F(P,-u))\}.
	\]
	Thus, unless \( \|\cR u\|_B=0 \), \( g_{P,\phi} \) determines \( \min\{\len (F(P,u)),\len(F(P,-u))\} \). This together with the information contained in \eqref{sum_lengths_opposite_edges} gives \( \{\len (F(P,u)), \len (F(P,-u))\} \).

	If \( \|\cR u\|_B=0 \), then \( l_{\cR u}\subset B \) and either \( B=\real^2 \) or \( B \) is an \( o \)-symmetric strip parallel to \( \cR u \).  Consider the case $B= \real^2$. In this case \( \phi=\alpha \vol \) and \( \alpha>0 \). It can be shown that
	\begin{equation} 
		\label{g vol asymp}
		g_{P}(x_0-\varepsilon u)
			=
		\min \{\len(F(P,u)), \len (F(P,-u))\} \eps + o(\eps),
			\quad \text{as} \ \eps \rightarrow +0,
	\end{equation}
	see \cite[proof of Lemma 3.1]{MR1909913}. Hence \( \min \{\len(F(P,u)), \len(F(P,-u))\} \) is determined by \( g_{P} \) and thus also by \( g_{P,\phi}= \alpha g_{P}. \) Now consider the remaining case, in which $B$ is an \( o \)-symmetric strip parallel to \( \cR u \). In this case
	\( \perim_B(\cdot) =\beta\wid(\cdot,u) \), for some known \( \beta \ge 0 \) (which is given by the knowledge of $B$).  Clearly, \( \gperim{P}(x_0-\varepsilon u) = \beta \eps\) for all sufficiently small $\eps>0$.
	Thus, taking into account \eqref{g vol asymp} we obtain
	\begin{equation*}
 		g_{P,\phi}(x_0-\varepsilon u)
			= 
		\Big(\beta + \alpha \min \{ \len(F(P,u)), \len(F(P,-u))\} \Big)\varepsilon + o(\varepsilon),
			\quad \text{as} \ \eps \rightarrow +0.
	\end{equation*}
	The strict monotonicity of \( \phi \) implies \( \alpha >0 \). Thus the  previous formula  determines
	\[
		\min\{\len(F(P,u)),\len(F(P,-u))\}
	\]
	and, as before, \( \{\len(F(P,u)), \len(F(P,-u))\} \).
\end{proof}

If both numbers in \( \{\len(F(P,u)), \len(F(P,-u))\} \) are strictly positive, then
\[
	\{\len(F(P,u)), \len(F(P,-u))\} = \{\ci(P,u), \ci(-P,u)\}.
\] 

\begin{claim}
	\label{determ vertex-edge claim}
	Assume that \( \len(F(P,u)) \) and \( \len(F(P,-u)) \) are not both zero.
	Then \( g_{P,\phi} \) determines \( \{ \ci(P,u),\ci(-P,u)\} \).
\end{claim}

\begin{proof}
	When both lengths are positive the assertion is a consequence of Claim~\ref{opposite lengths claim}. Assume that exactly one length vanishes.
	We may assume, up to a reflection, that \( F(P,u) \) is an edge and \( F(P,-u) \) is a vertex, say \( a \). Let the edges \( E_1 \) and \( E_2 \) of \( P \) containing \( a \) be contained in lines \(a+ l_1 \) and \(a+ l_2 \), and let \( F(DP,u)=[x_1,x_2] \).  Let the labeling and the point \( y\in DP \) be such that \( x_i\in y+l_i \), \( i=1,2 \). 
      Let \( m \) be a line parallel to \( [x_1,x_2] \) and intersecting the interior of the triangle  $\conv \{x_1,x_2,y\}$.  For all \( x\in m \) contained in the triangle  $\conv \{x_1,x_2,y\}$, \( g_{P,\phi} \) has the same value because $P\cap(P+x)$ changes only by a translation. For \( x\in m \) outside this triangle, \( g_{P,\phi} \) is less than this value, by the strict monotonicity of \( \phi \). Therefore the directions of the lines \( l_1 \) and \( l_2 \) can be determined.  This yields the outer normals of the edges \( E_1 \) and \( E_2 \) and hence the normal cone \( N(P,a) \).
\end{proof}

\begin{claim}
	\label{determ vertex-vertex claim}
	Assume \( \length (F(P,u))=\length (F(P,-u))=0 \).
	Then \( g_{P,\phi} \) determines \( \{\ci(P,u),\ci(-P,u)\} \).
\end{claim}

\begin{proof}
	Let \( F(P,u) = \{a_1\} \) and \( F(P,-u) = \{a_2\} \).  Then $\{ \ci(P,u), \ci(-P,u)\} = \{N(P,a_1), - N(P,a_2)\}$. Thus, we need to determine the set of the two cones $N(P,a_1)$ and $-N(P,a_2)$. We can argue exactly as in \cite[Case~2 of Lemma~3.1]{MR1909913} and in order to keep the presentation self-contained we repeat the argument.
	Let $i \in \{1,2\}$. If there exists $w \in \ucircle$ such that $F(P,w)=\{a_i\}$ and $F(P,-w)$ is an edge, then by Claim~\ref{determ vertex-edge claim} the cone $N(P,a_i)$ is determined by $g_{P,\phi}$, up to reflection in $o$.
	If by Claim~\ref{determ vertex-edge claim} both $N(P,a_1)$ and $-N(P,a_2)$ are determined using an appropriate direction $w \in \ucircle$ as above, the assertion follows.
	If precisely one of the two cones has been determined using $w \in \ucircle$, say the cone $-N(P,a_2)$, then for the other cone $N(P,a_1)$ one has the inclusion $N(P,a_1) \subset - N(P,a_2)$. Taking into account the known equality $N(DP,a_1-a_2) = N(P,a_1) \cap (-N(P,a_2))$, we obtain $N(DP,a_1-a_2) = N(P,a_1)$, which shows that also the cone $N(P,a_1)$ is determined. In the case that neither $N(P,a_1)$ nor $-N(P,a_2)$ can be determined using a direction $w \in \ucircle$ as above, we have $N(P,a_1) = -N(P,a_2)$ and, thus, both $N(P,a_1)$ and $-N(P,a_2)$ coincide with $N(DP,a_1-a_2)$. It follows that also in this case $N(P,a_1)$ and $-N(P,a_2)$ are determined by $g_{P,\phi}$.
\end{proof}
The proof of Lemma~\ref{curv info lem} is concluded.
\end{proof}

\begin{lemma}
	\label{symmetric arcs lem}
	Let \( \phi\in\Phi^2 \setminus \{0\} \) be strictly monotone, and let \( P \) and \( Q \) be convex polygons  with \( g_{P,\phi}=g_{Q,\phi} \) and such that \( P \) is not a reflection or a translation of $Q$. Let $A^+$ and $A^-$ be maximal arcs contained in $\bd P\cap\bd Q$ and assume that neither \( A^+ \) nor \( A^- \) are points. Assume also the existence of $u_0\in\ucircle$ such that $F(P,u_0)$ and $F(P,-u_0)$ are vertices of $P$ and
	\[
		F(P,u_0)\subset \relint A^+,\quad F(P,-u_0)\subset \relint A^-.
	\]
Then \( A^+ \) is a reflection of \( A^- \).
\end{lemma}
\begin{proof}Since \( P\neq Q \) neither \( A^+ \) nor \( A^- \) coincide with \( \bd P \). Let $a_1^+$ and $a_2^+$ denote, respectively, the left and right endpoint of $A^+$. Let $a_1^-$ and $a_2^-$ be defined similarly for $A^-$. For $i=1,2$, let $u_i^+$ be the unit outer normal to $P$ and $Q$ at the segment of $A^+$ containing $a_i^+$ and let $u_i^-$ be the unit outer normal to $P$ and $Q$ at the segment of $A^-$ containing $a_i^-$. We remark that $u_1^+\neq u_2^+$ and $u_1^-\neq u_2^-$, because both $\relint A^+$ and $\relint A^-$ contains a vertex, by assumption. Clearly $[u_1^+,u_2^+]_{\ucircle}$ is the set of unit outer normals to $P$ and $Q$ at  points in $\relint A^+$.
 
We claim that, for each $i=1,2$, the segment in $A^+$ containing $a_i^+$ is parallel to the segment in $A^-$ containing $a_i^-$, that is
\begin{equation}\label{lem_curv_info_aa}
 u_1^+=-u_1^- \quad \text{and}\quad u_2^+=-u_2^-.
\end{equation}
Let $u\in(u_1^+,u_2^+)_{\ucircle}$. We have
\begin{equation}\label{lem_curv_info_a}
 F(P,u)=F(Q,u)\subset\relint A^+.
\end{equation}
This  and~\eqref{faces of diff body} imply $F(P,-u)=F(Q,-u)$. This identity together with the fact that
\( \bigcup_{v\in (u_1^+,u_2^+)_{\ucircle}}F(P,-v) \)
is an arc (possibly, degenerate to a point) contained in $\bd P\cap \bd Q$ and intersecting $A^-$, imply
\begin{equation}\label{lem_curv_info_b}
F(P,-u)=F(Q,-u)\subset A^-.
\end{equation}
Formula~\eqref{lem_curv_info_a} implies $\ci(P,u)=\ci(Q,u)$ and, as a consequence of Lemma~\ref{curv info lem},
\[
 \ci(P,-u)=\ci(Q,-u).
\]
This and~\eqref{lem_curv_info_b} imply $F(P,-u)=F(Q,-u)\subset \relint A^-$. This implies $-u\in[u_1^-,u_2^-]_{\ucircle}$ and, for the arbitrariness of $u$,
\(
 -(u_1^+,u_2^+)_{\ucircle}\subset[u_1^-,u_2^-]_{\ucircle}.
\)
The analogous inclusion with the roles of $A^+$ and $A^-$ exchanged can be proved in a similar way. This concludes the proof of~\eqref{lem_curv_info_aa}.

Let \( u\in \ucircle \) be such that
\[
	(l_{u}+a_1^-) \cap\relint A^+ \neq\emptyset 
		\qquad \text{and} \qquad 
	(l_u+a_1^+) \cap\relint  A^- \neq\emptyset.
\]
Let \( r^-=\length(P\cap (l_u+a_1^-)) \) and \( r^+=\length(P \cap (l_u+a_1^+)). \)
We shall prove that \( r^- = r^+. \) 
Suppose that \( r^- \neq r^+, \) i.e., without loss of generality, that 
\[
r^-<r^+.
\] 

Let \( \{b\}=(l_u+a_1^+)\cap A^- \).
The boundaries of \( P \) and \( Q \) coincide in a neighborhood of \( b \). 
Let \( E_{PQ} \) be a segment with an endpoint in \( b \), contained in \( \bd P\cap\bd Q \) and outside the strip bounded by \( l_{u}+a_1^- \) and \( l_u+a_1^+ \).
The boundaries of \( P \) and \( Q \) differ in every neighborhood of \( a_1^+ \).
Let \( E_P \) and \( E_Q \) be segments with an endpoint in \( a_1^+ \), outside the strip bounded by \( l_{u}+a_1^- \) and \( l_u+a_1^+ \), and contained in \( \bd P \) and in \( \bd Q \), respectively.
Up to exchanging \( P \) and \( Q \) and reducing the lengths of \( E_P \) and \( E_Q \), we may assume that \( E_P\subset Q \), that is, all points of $P$ sufficiently close to \( a_1^+ \) belong to \( Q \).
 
Consider a chord \( [c_1,c_2] \) of \( P \), parallel to \( u \) with \( c_1\in E_{PQ} \) and \( c_2\in E_P \), and close enough to \( l_u+a_1^+ \) to ensure that \( r=\length([c_1,c_2])>r^- \).

\begin{figure}
	\begin{center}
		\input{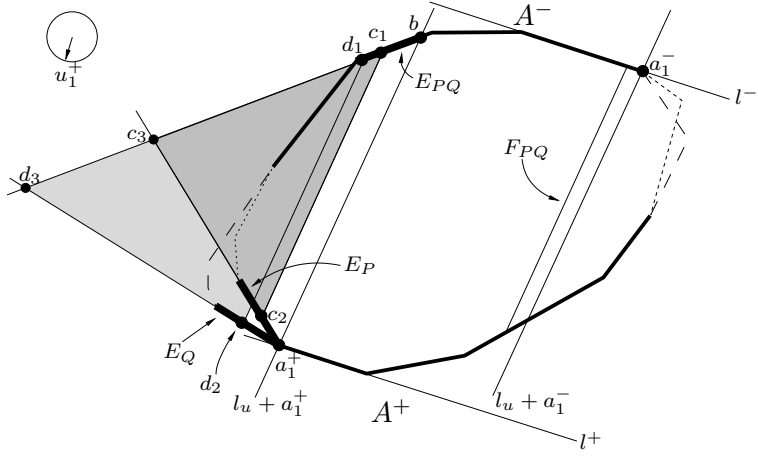}
	\end{center}
	\caption{
		The arcs \( A^+ \) and \( A^- \),
		the segments \( E_{PQ} \), \( E_P \), \( E_Q \) (thick segments) and \( F_{PQ} \), 
		the triangles \( \conv \{c_1,c_2,c_3\} \) and \( \conv \{d_1,d_2,d_3\} \) (in gray) 
		and the vector \( u_1^+ \).
	}
	\label{fig:arcs}
\end{figure}

By~\eqref{lem_curv_info_aa}, there is a line \( l^+ \) (and a line \( l^- \)) orthogonal to \( u_1^+ \) and supporting both \( P \) and \( Q \) at \( a_1^+ \) (at \( a_1^- \), respectively).
Let \( m \) be a supporting line to \( P \) at \( b \) and note that \( [c_1,c_2] \) lies between \( l^+ \) and \( m \), which are either parallel or meet in the half-plane bounded by \( l_u+a_1^+ \) not containing \( a_1^- \).
Since \( [c_1,c_2] \) is parallel to \( P\cap (l_{u}+a_1^+) \), we have \( r\le r^+ \), with equality if and only if \( c_2\in l^+ \), \( E_P\subset l^+ \) and \( c_1, b\in l^-=m \).
When equality holds, since \( l^+ \) supports \( Q \) too, the inclusion \( E_P\subset l^+ \) and the assumption \( E_P\subset Q \) imply \( E_Q\subset l^+ \),  which contradicts  the assumption \( A^+ \) maximal.
Therefore \( r<r^+ \).

Let us prove that \( E_{PQ} \) is not parallel to \( E_Q \).
If they are parallel, then, arguing as above, we have that \( E_{PQ}\subset l^-=m \) and \( E_Q\subset l^+ \).
Thus \( Q \) has two edges orthogonal to \( u_1^+ \).
By Lemma~\ref{curv info lem} the same happens for \( P \). We have \({F(P,u_1^+)}, {F(Q,u_1^+)}\subset l^+ \) 
and \( {F(P,-u_1^+)}, {F(Q,-u_1^+)}\subset l^- \).
The segment \( E_P \) is not contained in \( l^+ \), because this contradicts  the assumption \( A^+ \) maximal. Thus \( \len(F(Q,u_1^+))>\len(F(P,u_1^+)) \). Thus Lemma~\ref{curv info lem} implies
\[
\len(F(P,u_1^+))=\len(F(Q,-u_1^+))\text{ and } \len(F(P,-u_1^+))=\len(F(Q,u_1^+)).
\]
Since both $F(P,-u_1^+)$ and $F(Q,-u_1^+)$ contain $[a_1^-,b]$, then $F(P,u_1^+)$ and $F(Q,u_1^+)$ contain  a segment of length $\len ([a_1^-,b])$. This implies that  $l^+\cap (l_u+a_1^-)\in A^+$ and contradicts $r^-<r^+$.
This concludes the proof that \( E_{PQ} \) is not parallel to \( E_Q \).

If $[c_1,c_2]$ is sufficiently close to $l_u+a_1^+$, then  there is a chord \( [d_1,d_2] \) of \( Q \) which is  a translation of  $[c_1,c_2]$ and such that \( d_1\in E_{PQ} \) and \( d_2\in E_Q \) (see Figure~\ref{fig:arcs}).
Since \( r^-<r<r^+ \), there is a common chord \( F_{PQ} \) of \( P \) and \( Q \) of length \( r \), parallel to \( u \), contained in the strip bounded by \( l_u+a_1^+ \) and \( l_u+a_1^- \), and with endpoints on the arcs \( A^+ \) and \( A^- \). Let \( c_3=\aff(E_{PQ})\cap\aff(E_P) \) and \( d_3=\aff(E_{PQ})\cap\aff(E_Q) \).

Let \( x=c_1-c_2=d_1-d_2 \). In view of Theorem~\ref{radial deriv of g}, we have
\[
	-\left. \pderiv{t} g_{P}(t x)\right|_{t=1} 
		< 
	- \left. \pderiv{t} g_{Q}(t x)\right|_{t=1},
\]
since $\inspar(P,x)=\conv([c_1,c_2]\cup F_{PQ})$, $\inspar(Q,x)=\conv([d_1,d_2]\cup F_{PQ})$ and by this $\vol(\inspar(P,x))<\vol(\inspar(Q,x))$. Note that $\vol(\inspar(P,x))<\vol(\inspar(Q,x))$ holds because the line \( \aff [c_1,c_2] \) is closer to \( \aff F_{PQ} \) than the line \( \aff [d_1,d_2]. \) Furthermore, by Theorem~\ref{radial deriv of g perim} we have
\[
	\LeftPartialDerivAt{
		\Bigl(\gperim{Q}(t x)  - \gperim{P}(t x)  \Bigr)
	}{t}{1} 
		=
	\perim_B(\conv\{c_1,c_2,c_3\})-\perim_B(\conv\{d_1,d_2,d_3\}).
\]
By construction, the triangle \( \conv\{c_1,c_2,c_3\} \) is strictly contained in the translation of the triangle \(\conv\{d_1,d_2,d_3\}\) by vector \( c_2-d_2 \). Consequently
\[
	-\left. \pderivleft{t} \gperim{P}(t x) \right|_{t=1}
		\le 
	- \left. \pderivleft{t} \gperim{Q}(t x) \right|_{t=1},
\]
and the latter inequality is strict unless $\perim_B$ is not strictly monotone. By assumption, \( \phi = \alpha \vol + \perim_B \) is strictly monotone, and thus either $\perim_B$ is strictly monotone or \( \alpha > 0 \).  In both cases we arrive at the strict inequality
 \begin{equation}\label{contradiction}
-\left. \pderivleft{t} g_{P,\phi}(t x) \right|_{t=1} < \left. - \pderivleft{t} g_{Q,\phi}(t x) \right|_{t=1}.
\end{equation}
Inequality \eqref{contradiction} contradicts \( g_{P,\phi} = g_{Q,\phi} \).

It follows that \( r^-=r^+ \).
Therefore  \( (l_{u}+a_1^+)\cap A^- \) and \( (l_{u}+a_1^-)\cap A^+ \) are symmetric with respect to \( (a_1^++a_1^-)/2 \). Since we may repeat the above argument for every \( u \) such that \( l_{u}+a_1^- \) intersects \( \relint A^+ \) and \( l_u+a_1^+ \) intersects  \(  \relint A^- \) we have that either \( A^+ \) contains the reflection of \( A^- \) with respect to \( (a_1^++a_1^-)/2 \), or the same holds with the role of \( A^+ \) and \( A^- \) exchanged. 

Without loss of generality, assume that the reflection of $A^-$ with respect to $(a_1^+ + a_1^-)/2$ is a subset of $A^+$, that is $A_1^- := a_1^+ + a_1^- - A^- \subseteq A^+$. To conclude the proof, it remains to show the equality $A_1^- = A^+$. We argue by contradiction. Assume $A_1^-$ is a proper subset of $A^+$. Then $\len(A^-) < \len(A^+)$ and $A^-_1$ has two endpoints, one coinciding with the endpoint $a_1^+$ of $A^+$ and the other one $f_1:=a_1^+ + a_1^- - a_2^-$ lying in $\relint(A^+)$. Repeating the previous arguments with respect to points $a_2^+, a_2^-$ in place of $a_1^+, a_1^-$, we see that either the reflection of $A^-$ with respect to $(a_2^+ + a_2^-)/2$ is a subset of $A^+$ or the reflection of $A^+$ with respect to $(a_2^+ + a_2^-)/2$ is a subset of $A^-$. Since $\len(A^-) < \len(A^+)$, the former is the case, that is $A_2^- := a_2^+ + a_2^- - A^- \subseteq A^+$. The arc $A_2^-$ has two endpoints, one coinciding with the endpoint $a_2^+$ of $A^+$ and the other one $f_2 := a_2^+ + a_2^- - a_1^-$ lying in $A^+$. Since $A_1^-$ and $A_2^+$ coincide up to translations, the segments $[a_1^+, f_1]$ and $[a_2^+, f_2]$ joining the endpoints of $A_1^-$ and $A_2^-$, respectively, are parallel. Since $A^+$ is a convex arc which is not a segment and since $f_1 \in \relint A^+$, we conclude that no segment joining $a_2^+$ with a point of $A^+$ is parallel to $[a_1,f_1]$. Thus, $[a_1^+, f_1]$ and $[a_2^+, f_2]$ are not parallel, which is a contradiction. 
\end{proof}

\begin{proof}[Proof of Theorem~\ref{retrieval of polygons thm}]
This proof  coincides with the proof of \cite[Theorem~1.1]{MR1909913}, up to replacing references to Lemmas~3.1 and 4.1 in \cite{MR1909913}
with references to their analogs in this paper, i.e., to Lemmas~\ref{curv info lem} and \ref{symmetric arcs lem}, respectively.  We repeat here the proof for completeness.

Let $P$ be a planar convex
polygon and let $Q$ be a planar convex body with $\gphi{P}=\gphi{Q}$ and
$P\neq Q+\tau$, $P\neq -Q+\tau$ for each $\tau\in\real^2$.  Since \( DP=DQ=\supp \gphi{P} \) (by Lemma~\ref{g phi thm}~\eqref{g phi thm_a})
and $P$ is a polygon, $DQ$ and hence $Q$ must also be polygons.
We shall prove that both $P$ and $Q$ are centrally symmetric. Once that this is proved Theorem~\ref{retrieval centr sym thm} implies that $P=Q$, up to
translation, a contradiction.

To prove the central symmetry of $P$ and $Q$, let $a$ and $b$ be
 {\em opposite}
vertices of $P$, that is,
$$
\intr N(P,a)\cap\left(-\intr N(P,b)\right)\neq\emptyset.
$$
By Lemma~\ref{curv info lem} and $DP=DQ$
we may assume, after a translation and reflection of $Q$, if necessary,
that
$a$ and $b$ are also vertices of $Q$, and moreover $N(P,a)=N(Q,a)$ and
$N(P,b)=N(Q,b)$. We apply Lemma~\ref{symmetric arcs lem} with $A^+$ (and $A^-$) the maximal arc in $\bd P\cap \bd Q$ containing $a$ (containing $b$, respectively) and $u_0\in \intr N(P,a)\cap-\intr N(P,b)\cap\ucircle$.  The arcs $A^+$ and $A^-$ are not degenerate because when two
polygons have a vertex and the normal cone at that vertex in common, then
their boundaries must be equal in a neighborhood of that vertex.
Lemma~\ref{symmetric arcs lem} implies that $A^+$ is a reflection of $A^-$.  This yields
\begin{equation}\label{retrieval of polygons thm_a}
 N(P,a)=N(Q,a)=-N(P,b)=-N(Q,b).
\end{equation}
The validity of \eqref{retrieval of polygons thm_a} for all pairs of opposite vertices implies that all edges of $P$ come in parallel pairs and  that the same happens for $Q$. Let $[a_1,a_2]$ and $[b_1,b_2]$ be an arbitrary pair of parallel edges of $P$. It now suffices to show that these edges have the same length.  Let $a_1$, $a_2$, $b_1$, and $b_2$ be in
counterclockwise order in $\bd P$.
By Lemma~\ref{curv info lem} and $DP=DQ$, after possibly a translation and a reflection of $Q$,
$[a_1, a_2]$ and $[b_1,b_2]$ are also edges of $Q$ and thus $a_1,a_2,b_1$ and $b_2$ are also vertices
of $Q$. Keeping $Q$ henceforth fixed in this position it is clear that
both $a_1$, $b_1$ and $a_2$, $b_2$ are pairs of opposite vertices
(in the sense of the previous paragraph)
 of $P$ as well as of $Q$. This yields
 $N(P,a_1)=-N(P,b_1)=N(Q,a_1)=-N(Q,b_1)$ and
 $N(P,a_2)=-N(P,b_2)=N(Q,a_2)=-N(Q,b_2)$. Consequently  the boundaries
 of $P$ and $Q$ coincide also in a neighborhood of $[a_1, a_2]$ and
$[b_1,b_2]$. Then Lemma~\ref{symmetric arcs lem} shows that $[a_1, a_2]$ must be a reflection of $[b_1,b_2]$ and
so they have the same length.  This proves that both $P$ and $Q$ are
centrally symmetric.
\end{proof}

\subsection{
  Determination of polygons from the width-covariogram (Theorem~\ref{retrieval for width-covar thm}).
}

\label{section width-covariogram} 

In this section we assume \( \phi(K)=\wid(K,z) \),  for every convex body \( K \) and for some given fixed \( z\in \ucircle \). Moreover we use the symbol \( g_{K,w} \) for \( \gphi{K} \).

The width-covariogram has a simple expression in certain subsets of its support, and this expression  identifies these subsets.
Let us define the \emph{core} of \( K\in\cK^n_0 \) as
\begin{equation*} 
	\core K : = \left(F(K,z)-K\right)\cap\left(K-F(K,-z)\right).
\end{equation*}
See Fig.~\ref{fig:core}. Clearly $\core K$ depends on the choice of $z$. The next lemma implies that width-covariogram of \( K \) determines its core.

\begin{lemma}\label{lem:width_core} 
 Let $K\in\cK^n_0$ and \( x\in DK \). We have 
\begin{equation}\label{expression_wcov_in_core}
  \wcov{K}(x)=\wcov{K}(o)-\sprod{x}{z}
\end{equation}
if and only if \( x\in \core K \).
\end{lemma}

\begin{proof}
Observe that \eqref{expression_wcov_in_core} fails when \( \sprod{x}{z}<0 \) because in this case one has
\[
 	\wcov{K}(o) - \sprod{x}{z} > \wcov{K}(o) = \max_{y\in DK}\wcov{K}(y) \ge \wcov{K}(x).
\]
Moreover, \( \core K \) is contained in \(\{x : \sprod{x}{z}\geq0\} \) because both $F(K,z)-K$ and $K-F(K,-z)$ are contained in that half-space. As a consequence we may assume \( \sprod{x}{z}\geq0 \) to prove the equivalence.

The set \( K\cap(K+x) \) is contained in the strip $S$ bounded by the hyperplane $I_1$ orthogonal to \( z \) and supporting \( K \) at \( F(K,z) \), and by the hyperplane $I_2$ orthogonal to \( z \) and supporting \( K+x \) at \( F(K,-z)+x \). Since $\wid(S,z)$ equals \( \wid(K,z)-\sprod{x}{z} \) and \( \wcov{K}(o)=\wid(K,z) \), we have
\begin{equation*}
 \wcov{K}(x)= \wid\left( K\cap (K+x)\right)
 \leq \wcov{K}(o)-\sprod{x}{z},
\end{equation*}
 with equality holding if and only if $S$ is the minimal strip orthogonal to $z$ containing $K\cap(K+x)$.  This happen exactly when $I_1\cap K$ intersects $K+x$ and  $I_2\cap (K+x)$ intersects $K$, i.e. if and only if
\begin{equation*}
 F(K,z)\cap\left(K+x\right)\neq\emptyset,\quad\text{and}\quad \left(F(K,-z)+x\right)\cap K\neq\emptyset.
\end{equation*}
These conditions are equivalent, respectively, to \( x\in F(K,z)-K \) and to \( x\in K-F(K,-z) \).
\end{proof}

\begin{figure}
\begin{center}
\input{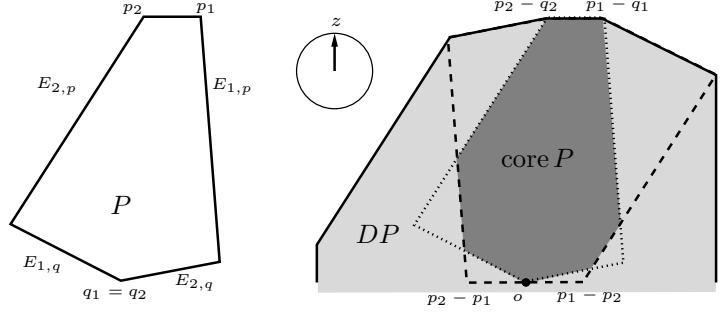}
\end{center}
\caption{The set \( \core P \) (dark gray) and a portion of \( DP \) (light gray). The figure depicts also \( P-F(P,-z) \) (bounded by a dotted line) and \( F(P,z)-P \) (bounded by a dashed line).}
\label{fig:core}
\end{figure}

Let us describe some properties of \( \core P \) for a planar convex polygon $P$ (see Fig.~\ref{fig:core}).
\begin{lemma}\label{lem:description_of_core}Let $P$ be a planar convex polygon and let \( F(P,z)=[p_1,p_2] \) and \( F(P,-z)= [q_1,q_2] \), where \( p_1,p_2,q_1,q_2 \) are in counterclockwise order on \( \bd P \).
 \begin{enumerate}[(I)]
  \item \label{lem:description_of_core_a} We have
 \begin{gather}
 F(\core P,z)=F(P,z)-F(P,-z)=[p_1-q_1,p_2-q_2];\label{face_core_z}\\
 \begin{split}\label{face_core_minusz}
  F(\core P,-z)&= D \big(F(P,z)\big)\cap D\big( F(P,-z)\big)\\
 &=[p_2-p_1, p_1-p_2]\cap[q_2-q_1, q_1-q_2].
 \end{split}
\end{gather}
\item \label{lem:description_of_core_b} Let \( E_{1,p} \) (and \( E_{1,q} \)) be the edge of \( P \) which precedes \( p_1 \) (and \( q_1 \), respectively) on \( \bd P \).  Let us consider  the edge of \( DP \) which precedes \( p_1-q_1 \) and the edge of \( \core P \) which precedes \( p_1-q_1 \). Then one of these edges is parallel to \( E_{1,p} \) and the other one is parallel to  \( E_{1,q} \).

\item \label{lem:description_of_core_c} Let \( E_{2,p} \) (and \( E_{2,q} \)) be the edge of \( P \) which follows \( p_2 \) (and \( q_2 \), respectively) on \( \bd P \). Let us consider  the edge of \( DP \) which follows \( p_2-q_2 \) and the edge of \( \core P \) which follows \( p_2-q_2 \). Then one of these edges is parallel to \( E_{2,p} \) and the other one is parallel to  \( E_{2,q} \).

\item \label{lem:description_of_core_d} If
\( F(P,z) \) is an edge and \( F(P,-z) \) is a vertex then $N(\core P,o)=N(P,q_1)$.
 \end{enumerate}
\end{lemma}
\begin{proof}
The set \( \bd P \) can be decomposed as the disjoint (except for the endpoints) union  of \( [p_1,p_2] \), \( [p_2,q_1]_{\bd P} \), \( [q_1,q_2] \) and \( [q_2,p_1]_{\bd P} \). Using this decomposition we can describe the boundaries of \( P-F(P,-z) \) and of \( F(P,z)-P \) as follows.
The set \( P^+:=P-F(P,-z) \) is bounded by the union of the arcs \( [p_1-q_1,p_2-q_2] \), \( [p_2,q_1]_{\bd P}-q_2 \), \( [q_1-q_2,q_2-q_1] \) and \( [q_2,p_1]_{\bd P}-q_1 \). The set \( P^-:=F(P,z)-P \) is bounded by the union of the arcs \( [p_2-p_1, p_1-p_2] \), \( p_2-[q_2,p_1]_{\bd P} \), \( [p_1-q_1,p_2-q_2] \) and \( p_1-[p_2,q_1]_{\bd P} \).

This description implies  \( F(P^+,z)=F(P^-,z)=[p_1-q_1,p_2-q_2] \), \( F(P^+,-z)=[q_1-q_2,q_2-q_1] \) and \( F(P^-,-z)=[p_2-p_1, p_1-p_2] \). Note that \( F(P^+,z) \) and \( F(P^-,z) \) are parallel and centered at \( o \). This proves \eqref{lem:description_of_core_a}.

When $p_1\neq p_2$ and $q_1=q_2$, then $F(P^-,-z)$ is an edge,   $F(P^+,-z)=o$  and $P^+\cap U=(P-q_1)\cap U$, for every small neighborhood $U$ of $o$. Thus we have $(\core P)\cap U=(P-q_1)\cap U$. This proves \eqref{lem:description_of_core_d}.

In order to prove \eqref{lem:description_of_core_b} and \eqref{lem:description_of_core_c} we observe that  \eqref{faces of diff body} implies
\begin{multline*}
 \{u\in\ucircle: \text{$F(DP,u)$ is an edge}\}=\{u\in\ucircle: \text{$F(P,u)$ is an edge}\}\\ \cup \{u\in\ucircle: \text{$F(-P,u)$ is an edge}\}.
\end{multline*}
Let \( \{u_1,u_2\} \) be the set consisting of  the unit outer normal vector to the edge \( E_{1,p} \) of \( P \) and of the unit outer normal vector to the edge \( -E_{1,q} \) of \( -P \). Label these vectors so that \( u_1 \), \( u_2 \) and \( z \) are on this order on \( \ucircle \). Then the edge of \( DP \) which precedes \( p_1-q_1 \) has outer normal vector \( u_2 \), while the edge of \( \core P \) which precedes \( p_1-q_1 \) has outer normal vector \( u_1 \). This proves \eqref{lem:description_of_core_b}, while \eqref{lem:description_of_core_c} can be proved analogously.
\end{proof}

Let us prove the equivalent of Lemma~\ref{curv info lem} for the width-covariogram.
\begin{lemma} \label{lem:curv_info_width}
Let \( \phi(\cdot)=\wid(\cdot,z) \), for some $z\in\ucircle$. Let \( P \) be a convex polygon in $\real^2$ and $u \in \ucircle$. Then \( \wcov{P} \) determines the set
	\begin{equation*}
		\{\ci(P,u), \ci(-P,u)\}.
	\end{equation*}
\end{lemma}
\begin{proof}
The proof of this lemma is divided into the proofs of Claims~\ref{cla:w_determination_opposite_lengths}, \ref{cla:w_determine_normal_cones_at_z}, \ref{cla:w_determination_vertex_edge} and~\ref{cla:w_determination_vertex_vertex}.

\begin{claim}\label{cla:w_determination_opposite_lengths}
For each \( u \in \ucircle \), \( \wcov{P} \) determines \( \{\length (F(P,u)), \length (F(P,-u))\} \).
\end{claim}

\begin{proof} This is proved as Claim~\ref{opposite lengths claim} except for the determination of
\[
 \min\{\length (F(P,z)), \length (F(P,-z))\}
\]
when \( u=z \) or \( u=-z \). This expression is determined by  \( \core P \), since it coincides with $(1/2)\len(F(\core P,-z))$, by \eqref{face_core_minusz}.
\end{proof}

\begin{claim}\label{cla:w_determine_normal_cones_at_z}
Let \( p_1 \), \( p_2 \), \( q_1 \) and \( q_2 \) be as in the statement of Lemma~\ref{lem:description_of_core}. Let \( C_1=N(P,p_1) \), \( C_2=N(P,p_2) \), \( D_1=N(P,q_1) \) and \( D_2=N(P,q_2) \). Then \( \wcov{P} \) determines \( \{C_1,-D_1\} \) and \( \{C_2,-D_2\} \).
\end{claim}

\begin{proof}
We recall that \( [p_1-q_1,p_2-q_2]=F(DP,z)=F(\core P,z) \) by~\eqref{faces of diff body} and ~\eqref{face_core_minusz}. Let \( \{u_1,u_2\} \) be the set consisting of  the unit outer normal vectors to the edge of \( DP \) which precedes \( p_1-q_1 \) and to the edge of \( \core P \) which precedes \( p_1-q_1 \). Let \( \{v_1,v_2\} \) be defined analogously as unit outer normals to the edges of  \( DP \) and $\core P$ which follow \( p_2-q_2 \). We distinguish three cases according to whether \( F(P,z) \) and \( F(P,-z) \) are edges or not.

Assume that both \( F(P,z) \) and \( F(P,-z) \) are edges. In this case $z$ is the right endpoint of  \( C_1\cap \ucircle \) and of \( (-D_1)\cap \ucircle \). The set of the left endpoints of these arcs coincide with \( \{u_1,u_2\} \), by Lemma~\ref{lem:description_of_core}~\eqref{lem:description_of_core_b}. Thus we have
\[
\{C_1\cap \ucircle,(-D_1)\cap\ucircle\}=\{[u_1,z]_{\ucircle}, [u_2,z]_{\ucircle}\}.
\]
A similar argument determines \( \{C_2,-D_2\} \).

Assume that exactly one among \( F(P,z) \) and \( F(P,-z) \) is an edge. We may assume, up to reflection, that the edge is \( F(P,z) \). Then
\[
 D_1=D_2=N(\core P,o),
\]
by Lemma~\ref{lem:description_of_core}~\eqref{lem:description_of_core_d}. The right endpoint of \( C_1\cap \ucircle \) is  \( z \). Its left endpoint is $u_1$, if $u_1=u_2$, or is the vector in \( \{u_1,u_2\} \) which is not left endpoint of $(-D_1)\cap\ucircle$,  if $u_1\neq u_2$. A similar argument determines \( \{C_2,-D_2\} \).

Assume that both \( F(P,z) \) and \( F(P,-z) \) are vertices. We have \( C_1=C_2 \) and \( D_1=D_2 \). 
The set of the left endpoints of $C_1\cap\ucircle$ and of $(-D_1)\cap\ucircle$ coincides with $\{u_1,u_2\}$, while the set of the right endpoints is $\{v_1,v_2\}$. If $v_1=v_2$ then
\[
	\{C_1\cap\ucircle,(-D_1)\cap\ucircle\}
		=
	\{[u_1,v_1]_{\ucircle}, [u_2,v_1]_{\ucircle}\}.
\] 
A similar formula holds when $u_1=u_2$. We may thus assume $u_1\neq u_2$ and $v_1\neq v_2$. Relabel  these vectors so that $\{u_1,u_2\}=\{\alpha_1,\alpha_2\}$, $\{v_1,v_2\}=\{\alpha_3,\alpha_4\}$ and  \( \alpha_1 \), \( \alpha_2 \), \( \alpha_3 \) and \( \alpha_4 \) are  in counterclockwise order on \( \ucircle \), with $z\in[\alpha_2,\alpha_3]_{\ucircle}$.
We may assume, after possibly replacing $P$ by $-P$, that \( \alpha_1 \) is the left endpoint of \( C_1\cap \ucircle \). We have to determine the right endpoint of $C_1 \cap \ucircle$. Let
\[
	x= -\varepsilon \cR \alpha_3,
\]
with \( \varepsilon>0 \) small enough (we recall that $\cR \alpha_3$ is the counterclockwise rotation of $\alpha_3$ by $90$ degrees), and let \( S \) be the minimal strip orthogonal to \( z \) and containing \( P\cap(P+x) \).
We distinguish two cases according to whether \( C_1\cap \ucircle=[\alpha_1,\alpha_4]_{\ucircle} \) or \( C_1\cap \ucircle=[\alpha_1,\alpha_3]_{\ucircle} \). Let \( E_{1,p} \), \( E_{2,p} \), \( E_{1,q} \) and \( E_{2,q} \) be as in the statement of Lemma~\ref{lem:description_of_core}.

\begin{figure}
\begin{center}
\input{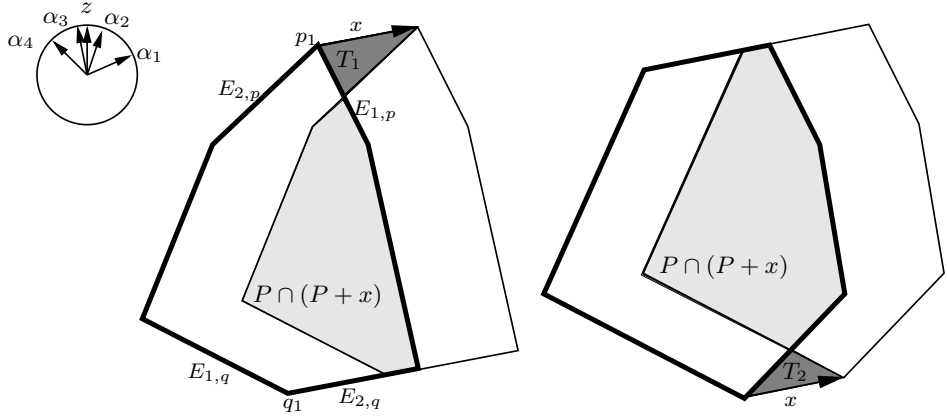}
\end{center}
\caption{\( P\cap(P+x) \) (light gray) when \( C_1\cap \ucircle=[\alpha_1,\alpha_4]_{\ucircle} \), on the left, and when \( C_1\cap \ucircle=[\alpha_1,\alpha_3]_{\ucircle} \), on the right. The triangles \( T_1 \) and \( T_2 \) are filled in dark gray.}
\label{fig:determination_cones_z}
\end{figure}

Assume \( C_1\cap \ucircle=[\alpha_1,\alpha_4]_{\ucircle} \). In this case \( (-D_1)\cap \ucircle=[\alpha_2,\alpha_3]_{\ucircle} \), \( E_{1,p} \), \( E_{2,p} \), \( E_{1,q} \) and \( E_{2,q} \) are orthogonal respectively to \( \alpha_1 \), \( \alpha_4 \), \( \alpha_2 \) and \( \alpha_3 \), see Fig.~\ref{fig:determination_cones_z}.
We have \( q_1+x\in P \) and thus one of the two lines  bounding \( S \) passes through \( q_1+x \). The other line bounding \( S \) contains the point \( E_{1,p}\cap(E_{2,p}+x) \). If we define
\[
 T_1 : =\conv\big\{p_1,p_1+x,E_{1,p}\cap(E_{2,p}+x)\big\},
\]
then we have
\begin{equation}\label{cla:w_determine_normal_cones_at_z_a}
 \wcov{P}(x)=\wid(P\cap(P+x),z)=\wid(P,z)-\wid(T_1,z).
\end{equation}
Assume \( C_1\cap \ucircle=[\alpha_1,\alpha_3]_{\ucircle} \). In this case \( (-D_1)\cap \ucircle=[\alpha_2,\alpha_4]_{\ucircle} \), \( E_{1,p} \), \( E_{2,p} \), \( E_{1,q} \) and \( E_{2,q} \) are orthogonal respectively to \( \alpha_1 \), \( \alpha_3 \), \( \alpha_2 \) and \( \alpha_4 \).
We have \( p_1\in P+x \) and thus one of the two lines  bounding \( S \) passes through \( p_1 \). The other line bounding \( S \) contains the point \( E_{2,q}\cap(E_{1,q}+x) \). If we define
\[
 T_2 : =\conv\big\{q_1,q_1+x,E_{2,q}\cap(E_{1,q}+x)\big\},
\]
then we have
\begin{equation}\label{cla:w_determine_normal_cones_at_z_b}
 \wcov{P}(x)=\wid(P\cap(P+x),z)=\wid(P,z)-\wid(T_2,z).
\end{equation}
 Both \( T_1 \) and \( T_2 \) have an edge equal to a translate of \( x \) and an edge orthogonal to \( \alpha_4 \). Since the third edge of \( T_1 \) is orthogonal to \( \alpha_1 \) while the third edge of \( T_2 \) is orthogonal to \( \alpha_2 \), the order between \( \alpha_1 \) and \( \alpha_2 \) implies that a translate of \( - T_2 \) is strictly contained in \( T_1 \) and  \( \wid(T_1,z)>\wid(T_2,z) \).

 The width-covariogram determines $\{u_1,u_2\}$ and $\{v_1,v_2\}$ and, through these vectors, $w(T_1,z)$ and $w(T_2,z)$.
It also determines $\wid(P,z)=\wcov{P}(o)$. It is thus possible to understand whether~\eqref{cla:w_determine_normal_cones_at_z_a} holds or~\eqref{cla:w_determine_normal_cones_at_z_b} holds and, through this choice, to decide whether \( C_1\cap \ucircle=[\alpha_1,\alpha_4]_{\ucircle} \) or \( C_1\cap \ucircle=[\alpha_1,\alpha_3]_{\ucircle} \).
\end{proof}

\begin{claim}\label{cla:w_determination_vertex_edge}
Assume that  \( \length (F(P,u)) \) and \( \length (F(P,-u)) \) are not both \( 0 \). Then \( \wcov{P} \) determines \(\{\ci(P,u), \ci(-P,u)\}\).
\end{claim}

\begin{proof}
When both lengths are positive the assertion is a consequence of Claim~\ref{cla:w_determination_opposite_lengths}. 
Assume that exactly one length vanishes.
We may suppose, up to  reflection, that \( F(P,u) \) is an edge and \( F(P,-u) \) is a vertex, say \( a \). In view of Claim~\ref{opposite lengths claim} it suffices to show that \( \wcov{P} \) determines \( N(P,a) \).

We distinguish two cases according to whether
\begin{equation}\label{condition_wclaim_2}
-u\in \intr C_1\cup\intr C_2\cup \intr D_1\cup\intr D_2
\end{equation}
or not.
By Claim~\ref{cla:w_determine_normal_cones_at_z}, the knowledge of $\wcov{K}$ makes it possible to determine the set of cones
\begin{equation}\label{cla:w_determine_normal_cones_at_z_ba}
 \{C_1,-C_1,D_1,-D_1,C_2,-C_2,D_2,-D_2\}.
\end{equation}
Since $u$ does not belong to the interior of any normal cone at a vertex of $P$ (because $F(P,u)$ is an edge, by assumption), \eqref{condition_wclaim_2} holds if and only if \( -u \) belongs to the interior of a cone in the set in \eqref{cla:w_determine_normal_cones_at_z_ba}.
Therefore the knowledge of $\wcov{K}$ makes it possible to understand whether \eqref{condition_wclaim_2} holds or not.

Assume that \eqref{condition_wclaim_2} does not hold.
Let us adopt the notations introduced in the proof of Claim~\ref{determ vertex-edge claim}. Let $T:=\conv \{x_1,x_2,y\}$. To determine $N(P,a)$ it suffices to determine $m_\varepsilon\cap T$.
As in Claim~\ref{determ vertex-edge claim},   \( \wcov{P}(x) \) is  constant when \( x\in m_\varepsilon\cap T \), because \( P\cap(P+x) \) changes only by a translation. Let \( x^\prime\in m_\varepsilon\cap T \) and \( x^{\prime\prime}\in m_\varepsilon\setminus T \), and let us prove that 
\begin{equation}\label{cla:w_determination_vertex_edge_aa}
 \wcov{P}(x^\prime)>\wcov{P}(x^{\prime\prime}).
\end{equation}
We remark that   a translation of \( P\cap(P+x^{\prime\prime}) \) is strictly contained in \( P\cap(P+x^{\prime}) \) and that, contrary to Claim~\ref{determ vertex-edge claim}, this inclusion alone it is not sufficient to show \eqref{cla:w_determination_vertex_edge_aa},
 because the width is not strictly monotone. 
Elementary arguments imply that in order to prove~\eqref{cla:w_determination_vertex_edge_aa} it suffices to prove that the boundary of the minimal strip orthogonal to \( z \) and containing  \( T \) intersects \( T \) only at \( x_1 \) and \( x_2 \). This is equivalent to prove that 
\begin{equation}\label{cla:w_determination_vertex_edge_a}
 z\notin N(T,y),\quad -z\notin N(T,y)\quad\text{ and }\quad z\neq \pm u.
\end{equation}
To prove $z, -z\notin N(T,y)$ we observe that $N(T,y)=N(P,a)$, by construction. 
If $\pm z\in N(P,a)$ then \( N(P,a) \) coincides, up to reflection, with $C_1$ or $C_2$ or $D_1$ or $D_2$, and this contradicts the assumption regarding~\eqref{condition_wclaim_2}, since $-u\in \intr N(P,a)$. 
The fact that $N(P,a)$ does not contain $z$ or $-z$ also implies \( u\neq z \) and \( u\neq-z \) (again because $-u\in \intr N(P,a)$). 

Assume that \eqref{condition_wclaim_2} hold.
If \( u=z \) we have  \( a=q_1=q_2 \)  and \( N(P,a)=D_1=D_2 \). Note that we have $p_1\neq p_2$ (because $F(P,u)$ is an edge, by assumption) and, as a consequence,  $C_1\neq C_2$. By Claim~\ref{cla:w_determine_normal_cones_at_z}, \( D_1 \) can be determined as the only cone in common to \( \{-C_1,D_1\} \) and \( \{-C_2,D_2\} \), where both $\{-C_1,D_1\}$ and $\{-C_2,D_2\}$ are determined by the $\phi$-covariogram.

When \( u=-z \) the argument is similar.
Assume \( u\neq z \) and \( u\neq -z \). Condition \eqref{condition_wclaim_2} implies  \( z\in N(P,a) \) or  \( -z\in N(P,a) \). This means that \( N(P,a) \) coincides  with either \( C_1 \) or \( C_2 \) or \( D_1 \) or \( D_2 \), because  these are the only normal cones at vertices of \( P \) containing \( z \) or \( -z \). We observe that among the eight cones in the union of \( \{C_1,-D_1\} \), \( \{C_2,-D_2\} \), \( \{-C_1,D_1\} \) and \( \{-C_2,D_2\} \) only one contains $-u$ in the interior, because  $F(P,u)$ is an edge.  Thus \( N(P,a) \) can be determined as the only cone in the union of \( \{C_1,-D_1\} \), \( \{C_2,-D_2\} \), \( \{-C_1,D_1\} \) and \( \{-C_2,D_2\} \) containing \( -u \) in its interior.
\end{proof}

\begin{claim}\label{cla:w_determination_vertex_vertex}
Assume  \( \length (F(P,u))=\length (F(P,-u))=0 \). Then \( \wcov{P} \)  determines \(\{\ci(P,u), \ci(-P,u)\}\).
\end{claim}

\begin{proof}
It coincides with the proof of Claim~\ref{determ vertex-vertex claim}.
\end{proof}
The proof of Lemma~\ref{lem:curv_info_width} is concluded.
\end{proof}

For the width-covariogram, Lemma~\ref{symmetric arcs lem} holds in a weaker form. 
The next two lemmas prove results which play for the width-covariogram the role played by Lemma~\ref{symmetric arcs lem} for the case of strictly monotone valuations.

\begin{lemma}\label{lem:w_symmetric_arcs}
Let \( P \), \( Q \), \( A^+ \), \( A^- \), \(u_0\), \( a_1^+ \), \( a_2^+ \), \( a_1^- \) and \( a_2^- \) be  as in Lemma~\ref{symmetric arcs lem}. Assume that neither $A^+$ nor $A^-$ are points or segments.
Let \( u\in \ucircle \) and \( i\in\{1,2\} \) be such that \( l_u+a_i^+ \) intersects \( \relint A^- \), and \( l_u+a_i^- \) intersects \( \relint A^+ \) (see Fig.~\ref{fig:w_arcs}).

Let \( S_P \)  and \( S_Q \) denote the minimal strips orthogonal to \( z \) and containing \( P \) and \( Q \), respectively.
Let \( S \) be the minimal strip orthogonal to \( z \) and containing the convex hull of the sub-arc of \( A^+ \) with endpoints \( a_i^+ \) and \( (l_u+a_i^-)\cap A^+ \) and of the sub-arc of \( A^- \) with endpoints \( a_i^- \) and \( (l_u+a_i^+)\cap A^- \).
\begin{enumerate}[(I)]
\item \label{cla:w_symmetric_arcs_I}If there exists \( v\in\{z,-z\} \) such that 
\begin{equation}\label{interior_to_union_of_strips}
F(S,v)\subset\intr(S_P\cup S_Q)
\end{equation}
then \( F(S,v) \) intersects one
of the two chords \( [a_i^+,(l_u+a_i^+)\cap A^-] \) and \( [a_i^-,(l_u+a_i^-)\cap A^+] \), and the length of the chord intersected by \( F(S,v) \) is less than or equal to the length of the other chord.
\item \label{cla:w_symmetric_arcs_II}If \( S\subset\intr(S_P\cup S_Q) \) then 
\begin{equation}\label{chords_equal_length}
 \len\left(  [a_i^+,(l_u+a_i^+)\cap A^-]\right)=\len\left(  [a_i^-,(l_u+a_i^-)\cap A^+]\right).
\end{equation}
\end{enumerate}
\end{lemma}

\begin{figure}
\begin{center}
\input{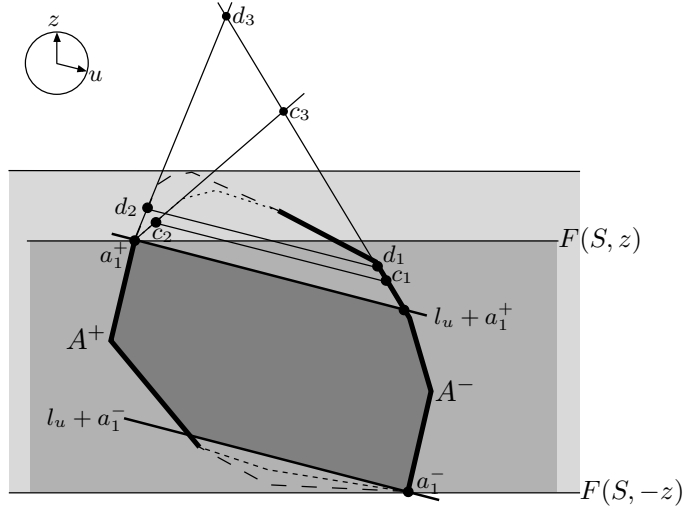}
\end{center}
\caption{The convex envelope of the sub-arcs (dark gray), the strips \( S \) (medium gray) and  \( S_P\cup S_Q \) (light gray). In this example~\eqref{interior_to_union_of_strips} holds when \( v=z \) and it does not hold when \( v=-z \).}
\label{fig:w_arcs}
\end{figure}
\begin{proof}
In order to prove~\eqref{cla:w_symmetric_arcs_I}, assume that \eqref{interior_to_union_of_strips} holds with \( v=z \). The line \( F(S,z) \) intersects one of the two chords in the statement because otherwise it intersects \( \conv\big([a_i^+,(l_u+a_i^-)\cap A^+]_{A^+}\cup[a_i^-,(l_u+a_i^+)\cap A^-]_{A^-}\big) \) at some point \( y\in\relint[a_i^+,(l_u+a_i^-)\cap A^+]_{A^+}\cup\,\relint[a_i^-,(l_u+a_i^+)\cap A^-]_{A^-} \). The convexity of the involved sets implies then that \( F(S,z) \) supports both \( P \) and \( Q \) at \( y \) and this contradicts~\eqref{interior_to_union_of_strips}.

Assume
\begin{equation}\label{def_intersection}
F(S,z)\cap[a_i^+,(l_u+a_i^+)\cap A^-]\neq\emptyset.
\end{equation}
 Let \( r^+=\len\big([a_i^+,(l_u+a_i^+)\cap A^-]\big) \), \( r^-=\len\big([a_i^-,(l_u+a_i^-)\cap A^+]\big) \) and assume \( r^+>r^- \). To prove that this inequality implies a contradiction, we follow closely the proof of Lemma~\ref{symmetric arcs lem}. Let \( c_i \) and \( d_i \), for \( i=1,2,3 \), be as in the proof of Lemma~\ref{symmetric arcs lem} (see Fig.~\ref{fig:w_arcs}). We recall some properties of these points.
\begin{enumerate}
 \item\label{inclusion_triangles} The triangles \( \conv\{c_1,c_2,c_3\} \) and \( \conv\{d_1,d_2,d_3\}+(c_1-d_1) \) are one strictly contained in the other and  have the  edge \( [c_1,c_2] \) in common.
 \item\label{lines_support}  The lines \( \aff([c_1,c_3]) \) and \( \aff([d_1,d_3]) \) coincide and support both \( P \) and \( Q \). The line \( \aff([c_2,c_3]) \) supports \( P \) and \( \aff([d_2,d_3]) \)  supports \( Q \).
 \item\label{arbit_close} Both \( [c_1,c_2] \) and  \( [d_1,d_2] \) can be chosen arbitrarily close to \( [a_i^+,(l_u+a_i^+)\cap A^-] \).
\end{enumerate}
We prove that
\begin{equation}\label{different_widths}
 \wid(\conv\{c_1,c_2,c_3\},z)\neq \wid(\conv\{d_1,d_2,d_3\},z).
\end{equation}
Choose a Cartesian coordinate system so that \( z=(0,1) \) and \( F(S,z) \) coincides with the \( x \)-axis.
It is evident that, given any \( p_1 \), \( p_2 \) and \( p_3\in\real^2 \), we have
\[
\wid(\conv\{p_1,p_2,p_3\},z)=\max\big(|\sprod{p_3-p_1}{z}|, |\sprod{p_3-p_2}{z}|, |\sprod{p_2-p_1}{z}|\big).
\]
The assumption $F(S,v)\subset\intr(S_P\cup S_Q)$ implies the existence of \( \alpha>0 \) such that the line \( l=\{p\in\real^2 : \sprod{p}{z}=\alpha\} \) supports \( P \) or \( Q \). Assume that \( l \) supports \( P \). 
Condition~\eqref{lines_support} and the convexity of \( P \) imply \( \sprod{c_3}{z}>\alpha \). 
On the other hand, \eqref{arbit_close} and the inclusion \( [a_i^+,(l_u+a_i^+)\cap A^-]\subset S \) imply \( \sprod{c_1}{z} <\alpha \) and \(\sprod{c_2}{z} <\alpha\). As a consequence we have \( \sprod{c_3-c_1}{z}>0 \), \( \sprod{c_3-c_2}{z}>0 \) and
\begin{equation}\label{form_width}
\wid(\conv\{c_1,c_2,c_3\},z)=\max\big( \sprod{c_3-c_1}{z}, \sprod{c_3-c_2}{z}\big).
\end{equation}
If \( \conv\{d_1,d_2,d_3\}+(c_1-d_1) \)  strictly contains \( \conv\{c_1,c_2,c_3\} \), then a formula similar to \eqref{form_width} holds for \( \wid(\conv\{d_1,d_2,d_3\}) \) and, moreover,
\[
\sprod{d_3+(c_1-d_1)}{z}>\sprod{c_3}{z}.
\]
This implies \( \wid(\conv(d_1,d_2,d_3),z)>\wid(\conv(c_1,c_2,c_3),z) \). If \( \conv(d_1,d_2,d_3)+(c_1-d_1) \)  is strictly contained in  \( \conv(c_1,c_2,c_3) \) then we have \( \sprod{d_3+(c_1-d_1)}{z}<\sprod{c_3}{z} \). This implies \( \wid(\conv(d_1,d_2,d_3),z)<\wid(\conv(c_1,c_2,c_3),z) \).  This concludes the proof of~\eqref{different_widths} when \( l \) supports \( P \). When \( l \) supports \( Q \), the proof is similar.

Let \( x=c_1-c_2 \).
In view of  Theorem~\ref{radial deriv of g perim}, we have
\begin{multline*}
-\left. \pderivleft{t} \wcov{P}(t x) \right|_{t=1}+\left. \pderivleft{t} \wcov{Q}(t x) \right|_{t=1}=  \\ =\wid(\conv\{c_1,c_2,c_3\},z)-\wid(\conv\{d_1,d_2,d_3\},z)\neq0
\end{multline*}
This contradicts \( \wcov{P}=\wcov{Q} \) and proves  \(r^+\leq r^- \) and~\eqref{cla:w_symmetric_arcs_I}.

In order to prove~\eqref{cla:w_symmetric_arcs_II} we observe that the assumption \( S\subset\intr(S_P\cup S_Q) \) implies that~\eqref{def_intersection} holds both when $v=z$ and when $v=-z$. Since \( F(S,z) \) and \( F(S,-z) \) intersect different chords, the lengths of these chords are equal, by~\eqref{cla:w_symmetric_arcs_I}.
\end{proof}

\begin{lemma}\label{lem:w_symmetric_arcs_conclusion}
Let \( P \), \( Q \), \( A^+ \), \( A^- \) and \(u_0\) be  as in Lemma~\ref{symmetric arcs lem}.
Let \( S_P \)  and \( S_Q \) denote the minimal strips orthogonal to \( z \) and containing \( P \) and \( Q \), respectively.
Assume that  neither \( A^+ \) nor \( A^- \) are points or segments.
\begin{enumerate}[(I)]
\item\label{cla:w_symmetric_arcs_conclusion_I}If \( S_P\neq S_Q \) then \( A^+ \) is a reflection of \( A^- \).
\item\label{cla:w_symmetric_arcs_conclusion_II} Assume \( S_P=S_Q \). If $\relint A^+\subset\intr S_P$ then \( A^+ \) contains a reflection of \( A^- \) or \( A^- \) contains a reflection of \( A^+ \). If $\relint A^+\cap\bd S_P\neq\emptyset$ then each component of \( A^+\cap\intr S_P \) is a reflection of  a component of \( A^-\cap\intr S_P \).
\end{enumerate}
\end{lemma}
\begin{proof}
Assume \( S_P\neq S_Q \).  The equality \( \wcov{P}(o)=\wcov{Q}(o) \) implies that \( S_P \) and \( S_Q \) have the same width in direction \( z \). Thus \( S_P\neq S_Q \) implies
\begin{equation}\label{sp_different_sq}
 S_P\cap S_Q\subset\intr(S_P\cup S_Q).
\end{equation}
Since \( S\subset S_P\cap S_Q \), Lemma~\ref{lem:w_symmetric_arcs} implies
\begin{equation}\label{equality_lengths_chords}
\len\left(  [a_1^+,(l_u+a_1^+)\cap A^-]\right)=\len\left(  [a_1^-,(l_u+a_1^-)\cap A^+]\right).
\end{equation}
The validity of this equality for each \( u\in\ucircle \) such that \( l_u+a_1^+ \) intersects \( \relint A^- \) and \( l_u+a_1^- \) intersects \( \relint A^+ \) implies that a sub-arc of \( A^+ \) is a reflection of \( A^- \) with respect to \( (a_1^++a_1^-)/2 \), or that the same hold with \( A^+ \) and \( A^- \) exchanged. A similar property can be proved for the symmetry with respect to \( (a_2^++a_2^-)/2 \). The two symmetries, together with the assumption that \( A^+ \) and \( A^- \) are not parallel segments, imply that \( A^+ \) is a reflection of \( A^- \). This proves~\eqref{cla:w_symmetric_arcs_I}.

Assume \( S_P=S_Q \). 
Arguing as we have done in the  proof of Lemma~\ref{symmetric arcs lem} we may prove that, for \( i \in \{1,2\}\), the segment of \( A^+ \) whose endpoint is  \( a_i^+ \) is parallel to the segment of \( A^- \) whose endpoint is \( a_i^- \).

Let $i\in\{1,2\}$ and let  us prove that
\begin{equation}\label{lem:w_symmetric_arcs_conclusion_a}
 a_i^+\in\intr S_P\quad\text{  if and only if }\quad a_i^-\in\intr S_P.
\end{equation}
Assume \( a_1^+\in\intr S_P \).
The segment contained in \( A^+ \)  whose endpoint is  \( a_1^+ \) and the one contained in $A^-$ whose endpoint is  \( a_1^- \) are not orthogonal to \( z \) because otherwise the lines containing them  define a strip containing $P$ and strictly contained in $S_P$, contradicting the definition of $S_P$.
Thus the lines through these segments define a strip  which intersects \( S_P \) in a parallelogram $E$ containing and supporting both \( P \) and \( Q \).
Let \( E_i \), \( i \in \{1,2,3,4\} \), denote the edges of this parallelogram, in counterclockwise order, with \( E_2\subset F(S_P,z) \) and \( E_4\subset F(S_P,-z) \).
Up to a reflection of $P$ and $Q$,
 we may assume \( a_1^+\in E_1 \) and \( a_1^-\in E_3 \).
Since \( E_3 \) contains a segment of \( A^- \) whose left endpoint is \( a_1^- \), we have \( a_1^-\neq E_3\cap E_4 \). Let us prove
\begin{equation}\label{lem:w_symmetric_arcs_conclusion_ac}
 a_1^-\neq E_2\cap E_3.
\end{equation}
Assume \eqref{lem:w_symmetric_arcs_conclusion_ac} false.
Let $w\in\ucircle$ be an outer normal to the parallelogram $E$ at $E_3$. We have
\begin{equation}\label{lem:w_symmetric_arcs_conclusion_ad}
 z, w\in N(P,a_1^-)\cap N(Q,a_1^-),
\end{equation}
because \( a_1^-\in E_2\subset F(S_P,z) \) and because $E_3$ supports both $P$ and $Q$ at \( a_1^-\).
The  cones \( N(P,a_1^-) \) and \( N(Q,a_1^-) \) are different, because $P$ and $Q$ are polygons which differ in every neighborhood of \( a_1^- \). Lemma~\ref{lem:curv_info_width} implies the existence of a vertex $b$ of $P$ and $Q$ such that
\begin{equation}\label{lem:w_symmetric_arcs_conclusion_ab}
 N(P,b)=-N(Q,a_1^-) \quad\text{and}\quad N(Q,b)=-N(P,a_1^-).
\end{equation}
Conditions~\eqref{lem:w_symmetric_arcs_conclusion_ad} and \eqref{lem:w_symmetric_arcs_conclusion_ab} imply  \[
-z, -w\in N(P,b)\cap N(Q,b).
\]
This implies $b\in E_1\cap E_4$.
Since $a_1^+$ is the left endpoint of a segment contained in $\bd P\cap \bd Q\cap E_1$, we have $a_1^+=b$. This contradicts the assumption \( a_1^+\in\intr S_P \), proves ~\eqref{lem:w_symmetric_arcs_conclusion_ac} and one of the implications of~\eqref{lem:w_symmetric_arcs_conclusion_a} when $i=1$.
The proof of the other implication and that of~\eqref{lem:w_symmetric_arcs_conclusion_a} when $i=2$ are completely analogous.

We observe that neither \( A^+ \) nor \( A^- \) intersect both lines bounding \( S_P \).
Indeed, if this is false then we have $F(P,v)=F(Q,v)$ for each $v\in(-z,z)_{\ucircle}$ or for each $v\in(z,-z)_{\ucircle}$. In each case this property and $DP=DQ$ imply \( P=Q \), by~\eqref{faces of diff body},  which contradicts the assumptions of the lemma.
We may thus assume \( a_i^-, a_i^+\in\intr S_P \), for some $i\in\{1,2\}$, say for $i=1$.

Assertion~\eqref{lem:w_symmetric_arcs_conclusion_a} together with the parallelism of the segment of \( A^+ \) whose endpoint is  \( a_2^+ \) and  the segment of \( A^- \) whose endpoint is \( a_2^- \), imply that
\[
A^+\cap\bd S_P=\{a_2^+\}\quad\text{ if and only if }\quad A^-\cap\bd S_P= \{a_2^-\}.
\]
We are thus in one of the following cases:
\begin{enumerate}[(i)]
\item\label{lem:w_symmetric_arcs_conclusion_b} \( A^+ \subset \intr S_P \) and \( A^-\subset \intr S_P \);
\item\label{lem:w_symmetric_arcs_conclusion_c} \( A^+\setminus\{a_2^+\}\subset\intr S_P \), \( A^-\setminus\{a_2^-\}\subset\intr S_P \)  and \( a_2^+,a_2^-\in\bd S_P \);
\item \label{lem:w_symmetric_arcs_conclusion_d} both \( \relint A^+ \) and \( \relint A^- \) intersects \( \bd S_P \).
\end{enumerate}

Arguments similar to those used to prove Assertion~\eqref{cla:w_symmetric_arcs_conclusion_I} of this lemma prove that~\eqref{lem:w_symmetric_arcs_conclusion_b} implies that \( A^+ \) is a reflection of \( A^- \),  while~\eqref{lem:w_symmetric_arcs_conclusion_c} implies that either a reflection of \( A^+ \) is contained in \( A^- \) or a reflection of \( A^- \) is contained in \( A^+ \). 

It remains to deal with Case~\eqref{lem:w_symmetric_arcs_conclusion_d}. We prove that in this case the component of \( A^+\cap\intr S_P \) containing \( a_1^+ \) is a reflection of  the component of \( A^-\cap\intr S_P \) containing \( a_1^- \). The corresponding result for the components containing \( a_2^+ \) and \( a_2^- \) is proved similarly.

Let \( b^+ \) (and let \( b^- \)) be the right endpoint of the component of \( A^+\cap\intr S_P \) containing $a_1^+$ (and of the component of \( A^-\cap\intr S_P \) containing $a_1^-$, respectively). We have \( b^+ \), \( b^-\in\bd S_P \). Start with \( u\in\ucircle \) equal to the direction \( v \) of \( a_1^--a_1^+ \) and increase \( u \) in counterclockwise direction. If \( u \) is close to \( v \) then
\begin{equation}\label{intersection_relint_arcs}
 (l_u+a_1^-)\cap\relint \big([a_1^+,b^+]_{A^+}\big)\neq\emptyset\quad\text{and}\quad(l_u+a_1^+)\cap\relint \big([a_1^-,b^-]_{A^-}\big)\neq\emptyset.
\end{equation}
If the strip $S$ is defined as in the statement of Lemma~\ref{lem:w_symmetric_arcs}, with \( i=1 \), then $S\subset \intr S_P$. By Lemma~\ref{lem:w_symmetric_arcs}, we have~\eqref{equality_lengths_chords}.  When we increase \( u \), the conditions \eqref{intersection_relint_arcs} are valid until \( b^+\in l_u+a_1^- \) or \( b^-\in l_u+a_1^+ \). Let \( w \) be the first \( u \) such that this happens, and assume, without loss of generality,  \( b^+\in l_w+a_1^- \). Let \( c^-=(l_w+a_1^+)\cap A^- \). We have \( c^-\in [a_1^-,b^-]_{A^-} \) and  \( [a_1^+,b^+]_{A^+} \) is a reflection of \( [a_1^-,c^-]_{A^-} \) with respect to \( (a_1^++a_1^-)/2 \). To conclude the proof it suffices to show that  \( c^-=b^- \). Assume the contrary, that is, assume \( c^-\in (a_1^-,b^-)_{A^-} \), and let \( v\in S^1 \) follow \( w \) in counterclockwise order and be so close to \( w \) so that
\begin{gather}
 (a_1^++l_v)\cap (c^-,b^-)_{A^-}\neq\emptyset, \label{za}\\
 (a_1^-+l_v)\cap (b^+,a_2^+)_{A^+}\neq\emptyset.\label{zb}
\end{gather}
Let \( S \) be defined as in the statement of Lemma~\ref{lem:w_symmetric_arcs}, with \(i=1\) and \( u=v \). Condition \eqref{za} implies that the line through \( (a_1^++l_v)\cap A^- \) and bounding \( S \) is contained in \( \intr S_P \). Therefore Lemma~\ref{lem:w_symmetric_arcs}~\eqref{cla:w_symmetric_arcs_I} implies
\begin{equation}\label{inequality_lengths_chords1}
\len\left(  [a_1^+,(l_v+a_1^+)\cap A^-]\right)\leq\len\left(  [a_1^-,(l_v+a_1^-)\cap A^+]\right).
\end{equation}
Let \( d^- \) be the reflection of \( (l_v+a_1^-)\cap A^+ \) with respect to \( (a_1^++a_1^-)/2 \). We have \( d^-\in l_v+a_1^+ \) and 
\begin{equation}\label{inequality_lengths_chords2}
\len\left(  [a_1^-,(l_v+a_1^-)\cap A^+]\right)=\len\left(  [a_1^+,d^-]\right).
\end{equation}
Simple geometric considerations imply that we also have
\( d^-\in\intr \conv\{a_1^+,c^-,b^-\} \) when \( v \) is sufficiently close to \( w \). Thus \( d^-\in\intr P \). This implies
\[
\len\left(  [a_1^+,d^-]\right)<\len\left( [a_1^+,(l_v+a_1^+)\cap A^-]\right).
\]
This inequality and \eqref{inequality_lengths_chords2} contradict \eqref{inequality_lengths_chords1}.
\end{proof}

\begin{proof}[Proof of Theorem~\ref{retrieval for width-covar thm}]
Let \( P \) be a planar convex polygon and let \( Q \) be a planar convex body with \( \wcov{P}=\wcov{Q} \).
Since \( DP=DQ=\supp \wcov{P} \) (by Lemma~\ref{g phi thm}~\eqref{g phi thm_a})
and \( P \) is a polygon, \( DQ \) and hence \( Q \) must also be polygons.
We shall prove that \( P=Q \), up to translations and reflections. Assume the contrary.

Let \( a \) and \( b \) be {opposite} vertices of \( P \), that is,
\[
\intr N(P,a)\cap(-\intr N(P,b))\neq\emptyset.
\]
By Lemma~\ref{lem:curv_info_width} and \( DP=DQ \)
we may assume, after a translation and reflection of \( Q \), if necessary,
that \( a \) and \( b \) are also vertices of \( Q \), and moreover
\( N(P,a)=N(Q,a) \) and \( N(P,b)=N(Q,b) \).
We show that when 
\begin{equation}\label{p_in_interior_of_strip}
\quad a\in\intr S_P\quad\text{or}\quad b\in\intr S_P
\end{equation}
then
\begin{equation}\label{opposite_cones_are_equal}
 N(P,a)=-N(P,b)=N(Q,a)=-N(Q,b).
\end{equation}
Assume~\eqref{p_in_interior_of_strip} and, say, \( a\in\intr S_P \). We apply Lemma~\ref{lem:w_symmetric_arcs_conclusion} with $A^+$ (and $A^-$) the maximal arc in $\bd P\cap \bd Q$ containing $a$ (containing $b$, respectively) and $u_0\in \intr N(P,a)\cap-\intr N(P,b)\cap\ucircle$. 
Neither \( A^+ \) nor  \( A^- \) are points, segments or are contained in the boundary of  $S_P$.
According to which conclusion of Lemma~\ref{lem:w_symmetric_arcs_conclusion}  holds true we have the following discussion. 
When \( A^- \) contains a reflection of \( A^+ \), and (since \( a\in\intr S_P \)) also when  each component of \( A^-\cap\intr S_P \) is a reflection of  a component of \( A^+\cap\intr S_P \), then \( \relint A^- \) contains a vertex \( c \) with $-u_0\in \intr N(P,c)$. Since $-u_0\in \intr N(P,b)$,  we have \( c=b \). When \( A^+ \)  contains a reflection of \( A^- \), then \( \relint A^+ \) contains a vertex \( d \) with $u_0\in \intr N(P,d)$. We conclude as before that \( d=a \).
In every case \( a \) and \( b \) are in the relative interior of symmetric arcs and this implies~\eqref{opposite_cones_are_equal}.

When there is no pair of opposite vertices \( a \) and \( b \) of \( P \) satisfying~\eqref{p_in_interior_of_strip} then \( P = \conv ( F(P,z) \cup F(P,-z) ) \).  By Lemma~\ref{lem:curv_info_width} and \( DP=DQ \), there is a translation and reflection of \( Q \) such that \( F(P,z)=F(Q,z) \) and  \( F(P,-z)=F(Q,-z) \). This implies \( P=Q \) and concludes the proof in this case.

When there are pairs of opposite vertices of \( P \) satisfying~\eqref{p_in_interior_of_strip}, the validity of~\eqref{opposite_cones_are_equal} for each such pair implies that  the edges of \( P \) nonorthogonal to \( z \) come in parallel pairs. 
Let \( a_1,\ldots,a_n, b_1,\ldots, b_n \) be  the vertices of \( P \) in counterclockwise order, with \( a_1,a_n, b_1 \) and \( b_n \) in \( \bd S_P \), all other vertices in \( \intr S_P \), and  \( [a_i,a_{i+1}] \) parallel to \( [b_i,b_{i+1}] \), \( i=1,\ldots,n-1 \). Note that \( a_1 \) may coincide with \( b_n \) and \( a_n \) may coincide with \( b_1 \). 
Let \( 2\leq i\leq n-2 \). As before,  after possibly a translation and a reflection of \( Q \), we may assume that \( [a_i,a_{i+1}] \) and \( [b_i,b_{i+1}] \) are also edges of \( Q \).
It is clear that both \( a_i \), \( b_i \) and \( a_{i+1} \), \( b_{i+1} \) are pairs of opposite vertices of \( P \). Since \( 1<i<n-2 \), these four vertices are contained in \( \intr S_P \).
This yields \( N(P,a_i)=-N(P,b_i)=N(Q,a_i)=-N(Q,b_i) \) and \( N(P,a_{i+1})=-N(P,b_{i+1})=N(Q,a_{i+1})=-N(Q,b_{i+1}) \).
Consequently  the boundaries of \( P \) and \( Q \) coincide also in a neighborhood of \( [a_i,a_{i+1}] \) and of \( [b_i,b_{i+1}] \).
Let $A^+$ (and $A^-$)  be the maximal arc in $\bd P\cap \bd Q$ containing $[a_i,a_{i+1}]$ (containing $[b_i,b_{i+1}]$, respectively) and $u_0\in \intr N(P,a_i)\cap-\intr N(P,b_i)\cap\ucircle$.
Each conclusion of Lemma~\ref{lem:w_symmetric_arcs_conclusion} implies that \( [a_i,a_{i+1}] \) is a reflection of \( [b_i,b_{i+1}] \). We remark that we use $[a_i,a_{i+1}]\subset\intr S_P$ in proving this claim.

We may assume, after possibly a translation and a reflection of \( Q \), that \( [a_1,a_{2}] \) and \( [b_1,b_{2}] \) are also edges of \( Q \).
What we have proved so far implies that 
\begin{equation*}\label{large_arcs}
[a_i,a_{i+1}] \quad \text{and}\quad [b_i,b_{i+1}],\quad i=1,\ldots,n-2
\end{equation*}
are edges both of $P$ and of $Q$. We are not able to conclude, in analogy to what we have done before, that \( \len([a_1,a_{2}])=\len([b_1,b_{2}]) \), because $a_1$, $b_1\in\bd S_P$ creates some difficulty in applying Lemma~\ref{lem:w_symmetric_arcs_conclusion}.
However, there is not enough freedom to have  $P\neq Q$. Indeed, by what we have proved so far and by  Lemma~\ref{lem:curv_info_width}, both $P$ and $Q$ have the following edges:  $[a_i,a_{i+1}]$ and $[b_i,b_{i+1}]$, $i=1,\ldots,n-2$, two edges parallel to $[a_{n-1},a_{n}]$  and zero, or one or two edges orthogonal to $z$ (according to whether $[a_n,b_1]$ and $[b_n,a_1]$ are edges or points). 
But there is only one convex polygon satisfying these conditions.
This implies \( P=Q \) and concludes the proof.
\end{proof}

\subsection{
	  Examples of nondetermination in dimension $n\geq3$
} 

\label{subsection nonuniqueness} 

Theorem~1.2 in~\cite{MR2108257} proves that, given $H \in \cK_0^\ell$ and $K \in \cK_0^m$, we have $g_{H \times K} = g_{H \times (-K)}$. It also proves that when neither $H$ nor $K$ are centrally symmetric then $H\times K$ is not a translation or a reflection of $H\times(-K)$. This construction allows to create pairs of convex bodies with equal covariogram which are not a translation or reflection of each other  in every dimension $n\geq4$.
Moreover these examples (together with their images under a linear map) are substantially the only known examples of nondetermination by the covariogram.
In the following theorem we show that the previous arguments extend directly to every valuation $\phi$ which is invariant with respect to the group of isometries of the Euclidean space $\real^n$.

\begin{theorem}\label{teo:nondetermination_dimension4}
	Let $K \in \cK_0^\ell$ and $H \in \cK_0^{m}$ and let $\phi : \cK^{\ell+m} \rightarrow \real$ be a valuation which is invariant with respect to the group of isometries of the Euclidean space $\real^n$.
	\begin{enumerate}[(I)]
	 	\item 
			\label{teo:nondetermination_dimension4_I} We have $\gphi{K \times H} = \gphi{K \times (-H)}$.
	 	\item 
			\label{teo:nondetermination_dimension4_II} For every $n\geq4$ there are pairs of convex bodies in $\real^n$ with equal $\phi$-covariogram which are not a translation or reflection of each other.
	\end{enumerate}
\end{theorem}
\begin{proof}
	Let us prove~\eqref{teo:nondetermination_dimension4_I}. For $K \in \cK^n$  we introduce the shorthand notation $K_x := K \cap (K + x)$. Let $x \in \real^m$ and $y \in \real^\ell$. We will show $\gphi{K \times H}(x,y) = \gphi{K \times (-H)}(x,y)$. Clearly, $(K \times H)_{(x,y)} = K_x \times H_y$ and thus $g_{K \times H}(x,y) = \phi(K_x \times H_y)$.
	Noticing that $K_x \times H_y$ can be transformed into $K_x \times (-H_y)$ by an isometry, we get $\gphi{K \times H}(x,y) = \phi(K_x \times (-H_y))$. The trivial relation $-H_y = (-H)_y - y$  implies $\gphi{K \times H}(x,y) = \phi( K_x \times (-H)_y - (o,y))$.
	Every translation is obviously an isometry, and so in the above expression the translation vector $-(o,y)$ can be discarded. We arrive at $\gphi{K \times H}(x,y) = \phi ( K_x \times (-H)_y ) = \gphi{K \times (-H)}(x,y)$.

	The proof of~\eqref{teo:nondetermination_dimension4_II} coincides with the corresponding one for the covariogram. 
\end{proof}

When $\phi$ is the width, similar counterexamples can be constructed in every dimension $n\geq3$.

\begin{theorem}\label{thm:cart_prod_for_wcov}
	Let \( H \in \cK^\ell_0 \), \( K \in \cK^m_0 \), $z=(o,z')\in \real^\ell\times\real^m$
	with $z'\in\usphere^{m}$ and let $\phi$ denote the  width in direction $z$.
	\begin{enumerate}[(I)]
	 \item Then \(\gphi{H \times K} \) is completely determined by \( DH \) and \( K \) by means of the following equality, which is valid for every \( (x,y)\in \real^\ell \times \real^m \):
	\[
		\gphi{H \times K}(x,y)= \chf_{DH}(x)\  \wid((K \cap (K+y)),z').
	\]
	\item If $H'\in\cK^\ell_0$ and $DH=DH'$, then \( \gphi{H \times K}=\gphi{H' \times K}\).
	\end{enumerate}
\end{theorem}
\begin{proof}
	 We have
	\[
		(H \times K) \cap (H \times K + (x,y)) = (H \cap (H+x)) \times (K \cap (K+y)).
	\]
	Thus, if \( x \not\in DH \), we have \( H \cap (H+x) = \emptyset \) and by this \( \gphi{H \times K}(x,y)=0 \). On the other hand, if \( x \in DH \), we have \( H \cap (H+x) \ne \emptyset \) and by this
	\begin{align*}
		\gphi{H \times K}(x,y)
		& = \wid((H \cap (H+x)) \times (K \cap (K+y)),(o,z'))
		\\ & = \wid((K \cap (K+y)),z').
	\end{align*}
\end{proof}

Theorem~\ref{thm:cart_prod_for_wcov} can be used to prove  Theorem~\ref{counterexamples_width_cov} by choosing $\ell \ge 2$, $H'$ a simplex, $H=(1/2) DH'$, $m=1$ and $K=[-1,1]$. We will give another proof of Theorem~\ref{thm:cart_prod_for_wcov}, which provides counterexamples with a different, much richer,  structure.
Let $z\in \usphere^{n-1}$. A set \( K \in \cK^n \) is called  \emph{$z$-prismatoid} with bases \( F(K,z) \) and \( F(K,-z) \) if \( K= \conv (F(K,z) \cup F(K,-z)) \).

\begin{theorem} Let $z\in \usphere^{n-1}$ and let $\phi$ be the width in direction $z$.
	\label{thm on prismatoids}
	\begin{enumerate}[(I)]
	 \item \label{cla:thm on prismatoids_I}Let \( K \in \cK^n_0 \) be a $z$-prismatoid with bases \( F = F(K,z) \) and $G = F(K,-z)$ and assume $DF= DG$. Then $\gphi{K}$ is
	 determined by $DF$ and $F-G$.
	 \item \label{cla:thm on prismatoids_II} Let $H, H'\subset\{x : \sprod{x}{z}=0\}$ and $L\subset\{x : \sprod{x}{z}=1\}$ be convex compact sets and assume $DH=DH'$. Then $K=\conv((H+L) \cup (H-L))$ and $K'=\conv((H'+L) \cup (H'-L))$ are $z$-prismatoids with the same $\phi$-covariogram.
	\end{enumerate}
\end{theorem}
\begin{proof}
	For showing Assertion~\eqref{cla:thm on prismatoids_I} it suffices to verify
	\begin{equation}
			DK = \conv \Bigl( (F - G) \cup (G - F) \cup DF \Bigr) \label{DK:decomp}
	\end{equation}
	and, for $x\in DK$, 
	\begin{equation}
		\gphi{K}(x) = \wid(K,z) - |\sprod{z}{x}|. \label{width:cov:very:simple}
	\end{equation}

	Taking into account \( K = \conv(F \cup G) \) and \( DF=DG \), equality \eqref{DK:decomp} is derived in the following straightforward way:
	\begin{align*}
		DK 
			& = \conv (F \cup G) - \conv(F \cup G) 
			\\ & = \conv\Bigl((F \cup G) - (F \cup G)\Bigl)
			\\ & = \conv\Bigl((F - G) \cup (G - F) \cup DF\Bigl),
	\end{align*}
	Here we used the identity $\conv DA = D\conv A$, which is  valid  for every  $A\subset\real^n$ (see \cite[Theorem 1.1.2]{Schneider-1993}). Let $\core K$ be defined as in the paragraph preceding Lemma~\ref{lem:width_core} and let us prove
	\begin{equation}
			DK  = \core K \cup (-\core K). \label{DK:core:eq}
	\end{equation}
	As soon as \eqref{DK:core:eq} is shown, \eqref{width:cov:very:simple} is a consequence of \eqref{DK:core:eq} and Lemma~\ref{lem:width_core}.
	We have \( \core K \cup (- \core K) \subset DK \) by definition of \( \core K \) and \( DK \). Thus, for concluding the proof it suffices to show \( DK \subset \core K \cup (- \core K )\).

	Let \( x \in DK\). By~\eqref{DK:decomp} and since $F-G$, $G-F$ and $DF$ are convex sets, \( x \) can be represented as a convex combination of three vectors \( x_1 \in F-G \), \( x_2 \in G - F \) and \( x_3 \in DF \), say \( x = \lambda_1 x_1 + \lambda_2 x_2 + \lambda_3 x_3 \) with \( \lambda_i \ge 0 \) for \( i \in \{1,2,3\} \) and \( \lambda_1 + \lambda_2 + \lambda_3 = 1 \). We distinguish between the case \( \lambda_1 \le \lambda_2 \) and the case \( \lambda_1 \ge \lambda_2 \). Consider the case \( \lambda_1 \ge \lambda_2 \). One has
	\begin{align*}
		x & = (\lambda_1-\lambda_2) x_1 + \lambda_2 (x_1+ x_2) + \lambda_3 x_3
			\\ & \in (\lambda_1 - \lambda_2) (F-G) + \lambda_2 (F - G + G - F) + \lambda_3 DF
			\\ & = (\lambda_1 - \lambda_2) (F-G) +  \lambda_2 (DF+DG) + \lambda_3 DF
			\\ & = (\lambda_1 - \lambda_2) (F-G) + 2 \lambda_2 DF + \lambda_3 DF
			\\ & = (\lambda_1 - \lambda_2) (F - G) + (2 \lambda_2 + \lambda_3) DF.
	\end{align*}
	Hence we obtain
	\begin{align*}
		x  \in & \conv ((F - G) \cup DF) 
			\\  = & \conv ((F - G) \cup (F - F))
			\\  = & \conv (F - (G \cup F))
			\\  = & F - \conv (G \cup F)
			\\  = & F - K.
	\end{align*}
	Here we used again \cite[Theorem 1.1.2]{Schneider-1993}. Using \( DF=DG \) in a similar fashion we obtain $x \in K - G$. Above we have shown \( x \in (F - K) \cap (K - G) = \core K \). Analogously, in the case \( \lambda_1 \le \lambda_2 \) it can be shown that \( x \in - \core K \). By this we obtain \eqref{DK:core:eq} and, thus, also \eqref{width:cov:very:simple}.
	
	For showing \eqref{cla:thm on prismatoids_II} we observe that the assumptions of Assertion~\eqref{cla:thm on prismatoids_I} are fulfilled because
	\begin{gather*}
	D(H+L)=D(H-L)=DH+DL, \\
	D(H'+L)=D(H'-L)=DH'+DL.
	\end{gather*}
Thus $\gphi{K}$ is uniquely determined by $D(H+L) = DH + DL$ and $(H+L) - (H-L) = DH + 2 L$. Consequently, $\gphi{K}$ is determined by $DH$ and $L$, that is, if we replace $H$ by  $H'$  the width-covariogram remains unchanged.
\end{proof}

\begin{proof}[Proof of Theorem~\ref{counterexamples_width_cov}]
It suffices to define $K$ and $K'$ following the construction described in Theorem~\ref{thm on prismatoids}~\eqref{cla:thm on prismatoids_II}. For instance, let $H'$ be an $(n-1)$-dimensional simplex in $\{x : \sprod{x}{z}=0\}$ and let $H=(1/2) DH'$. The set $H$ is $o$-symmetric and $DH=DH'$. Let $L$ be a noncentrally symmetric convex polytope in $\{x : \sprod{x}{z}=1\}$. We have $H+L\subset \{x : \sprod{x}{z}=1\}$ and  $H-L\subset \{x : \sprod{x}{z}=-1\}$. Moreover $H-L=-(H+L)$, and this implies that $K$  is $o$-symmetric.

The set $K$ is not a translation of $K'$  because $F(K,z)=H+L$ is not a translation of $F(K',z)=H'+L$. Indeed, if $H+L=H'+L+\tau$, for some $\tau\in\real^n$, then $H=H'+\tau$, by the cancellation law for Minkowski addition \cite[p. 126]{Schneider-1993}, and this identity is false. 
\end{proof}

\section{Random variables associated to \( \phi \)-covariograms} 
\label{sect rand var}

The measurements of random chords of a given set are discussed in    Ehlers and Enns \cite{MR0471018}, \cite{EnnsEhlers-1981}, \cite{MR1242019}, Santal\'o \cite[Chapter~4]{MR2162874} and Schneider and Weil \cite[Section~8.6]{MR2455326}.

We begin this section by presenting three random variables which provide the same information about $K$ as \( g_K \).

The first one has been considered by Matheron \cite{MR0385969} and  Nagel \cite{MR1232748}.
Let  \( K\in\cK^n \), \( u\in \usphere^{n-1} \), and let \( l \) be a random line parallel to \( u \)  distributed uniformly among all lines parallel to \( u \) that intersect \( K \).
 This random variable is defined  by
\[
 L_{\mu,u}=\length (l\cap K).
\]
If we change the definition of \( L_{\mu,u} \) by letting also \( u \) to be chosen at random on \( \usphere^{n-1} \), then we get \( L_\mu \), that is the length of a chord chosen under \emph{\( \mu \)-randomness} \cite{MR0471018}.  

The second random variable has been considered by Adler and Pyke \cite{AP91} and is defined as  \( X_1-X_2 \), where \( X_1 \) and \( X_2 \) are  independent random variables uniformly distributed in \( K \).

The third random variable is defined by
	\[
 		L_{\nu,u}
			=
		\length\left((X+l_u)\cap K\right),
	\]
	where $X$ is a random variable uniformly distributed in $K$. It corresponds to choosing the chord of \( K \) under \emph{\( \nu \)-randomness} \cite{MR0471018}. 
	
Knowing the distribution of \( L_{\mu,u} \) for each \( u \) or knowing the distribution of $X_1-X_2$ is equivalent to knowing \( g_K \) (see, for instance, \cite{averkov-bianchi-2009}).  The same holds true for \( L_{\nu,u} \) too: the knowledge of the distribution of $L_{\nu,u}$ for each \( u \) is equivalent to the knowledge of \( g_K \).  Since we have not found this mentioned in the literature, we prove it.  For each \( r\geq0 \)  the event \( \{L_{\nu,u}\geq r\} \) coincides with the event \( \{X\in A\} \), where \( A \) is the  union of all chords of \( K \) parallel to \( u \) and of length at least \( r \). Let $A_u$ be the orthogonal projection of $A$ onto the orthogonal complement of $u$. It is known that \( -\pderiv{r}g_K(ru) \) depends continuously on $r$ for $0< r < \rho(DK,u)$ and coincides with the \( (n-1) \)-volume of \( A_u \); see \cite[Proposition~4.3.1]{MR0385969}.
	Consequently, \( \vol(A) = g_K(ru)-r\pderiv{r}g_K(ru) \). Thus we have
	\begin{equation}\label{diff_equation_for_cov}
 		\Prob(L_{\nu,u}\geq r)
			=
		\frac{g_K(ru)}{\vol(K)}-r\pderiv{r}\left(\frac{g_K(ru)}{\vol(K)}\right),
	\end{equation}
where the notation \( \Prob \) stands for the probability of a random event. 
This formula shows that the knowledge of \( g_K \) gives the distribution of \( L_{\nu,u} \) for each \( u \) (recall that \( g_K(o)=\vol(K) \)). 
On the other hand, formula~\eqref{diff_equation_for_cov} is a differential equation for \( g_K(ru)/\vol(K) \). 
The distribution  of \( L_{\nu,u} \), for a given \( u \), determines $\rho(DK,u)$, because the support of this distribution is $[0,\rho(DK,u)]$. The right hand side of \eqref{diff_equation_for_cov} can be rewritten as $-r^2 \pderiv{r} \left( \frac{g_K(ru)}{r \vol(K)}  \right)$  for $0<r<\rho(DK,u)$. Hence $g_K(ru)/\vol(K)$ for $r \in [0,\rho(DK,u)]$ can be determined by the knowledge of $\Prob(L_{\nu,u} \ge r)$ for $r \in [0,\rho(DK,u)]$ by means of integration, by taking into account that $g_K(ru)$ vanishes for $r=\rho(DK,u)$. This determines  \( g_K(x)/\vol(K) \)  for each \( x\in\real^n \).  
On the other hand, the integral of $g_K / \vol(K)$ on \( \real^n \) equals \( \vol(K) \); see  Theorem~\ref{g phi thm} (\ref{integral of g phi}). 
We can thus determine \( g_K \).

Let us now pass to  random variables related to \( \phi \)-covariograms for $\phi$ more general than the volume. Let us start by proving Theorem~\ref{thm random lambda sect}.  Ehlers and Enns \cite{EnnsEhlers-1981} study \( L_{\gamma,u} \) in the case of $\len_B$ being the Euclidean length.
These authors denote the  way of choosing a random chord of \( K \) which corresponds to \( L_{\gamma,u} \) as  \emph{\( \gamma \)-randomness}.

\begin{proof}[Proof of Theorem~\ref{thm random lambda sect}] 
	We prove that for $r \ge 0$ we have 
		\begin{equation}\label{dist_for_l3}
			\Prob( L_{\gamma,u}\ge r ) = 
			\begin{cases}
				1 & \text{if} \ 0 \le r \le r_1,
				\\ \big(\gperim{K}(r u) + r \|u\|_B\big)/\perim_B(K) &
				\text{if} \ r_1 < r \le r_2,
				\\ \gperim{K}(r u)/\perim_B(K) & \text{if}  \ r_2 < r,
			\end{cases}
		\end{equation}
	where 
	\begin{align*}
		r_1 & := \min \{ \len(F(K,\cR u)), \len(F(K,-\cR u)) \},
		\\ r_2 & := \max \{ \len(F(K,\cR u)), \len(F(K,-\cR u)) \}.
	\end{align*}
	
	The case $0 \le r \le r_1$ of \eqref{dist_for_l3} is trivial since every chord of $K$ parallel to $u$ has length at least $r_1$. In the case  $r_2 < r$ the formula holds because in this case the event \( \{L_{\gamma,u}\geq r\} \) coincides with the event \( \{Y \not\in \relint \arc(r u) \cup \relint  \arc(-ru)\} \) (we use the notations introduced at the beginning of Section~\ref{sect deriv}), which has probability $\gperim{K}(r u) / \perim_B(K)$. Consider the case $r_1 < r \le r_2$. In this case the parallelogram $\inspar(r u)$ has exactly one edge parallel to $u$ and lying in the boundary of $K$. Without loss of generality, assume  $[p_3(ru),p_4(ru)] \subset \bd K$, that is, $[p_3(ru),p_4(ru)]=\arc(-ru)$. In this case $\{ L_{\gamma,u} \ge r \} = \{ Y \not\in \relint \arc(r u)\}$. The event $\{Y \not\in \relint \arc(ru)\}$ is the disjoint union of the events $\{ Y \not\in \relint(\arc(ru)) \cup \relint(\arc(-ru)) \}$ and $\{Y \in [p_3(ru),p_4(ru)] \}$, which have probabilities $\gperim{K}(r u) / \perim_B(K)$ and $r \| u\|_B / \perim_B(K)$, respectively. This yields \eqref{dist_for_l3} in the case $r_1 < r \le r_2$. 
		
	The knowledge of $B$ and \( \gperim{K} \) determines \( \perim_B(K) = \gperim{K}(o) \) and the values $r_1$ and $r_2$ (by Claim~\ref{opposite lengths claim} for the direction $\cR u$).  Thus \eqref{dist_for_l3} shows that the knowledge of $B$ and $\gperim{K}$ determines the distribution of \( L_{\gamma,u} \).  
	
	For the converse implication, we assume that $B$ and the distribution of $L_{\gamma,u}$ is known for every $u \in \ucircle$. This yields $\rho(DK,u)$ for every $u \in \ucircle$ and determines $DK$. Using the knowledge of $B$ we also determine $\perim_B(K) = \frac{1}{2} \perim_B(DK)$. Having $\perim_B(K)$, the $\perim_B$-covariogram is determined from \eqref{dist_for_l3} at every vector $r u$    with $r>0$ and $u \in \ucircle$ whenever $r_1=r_2=0$. Note that $r_1=r_2=0$ if and only if $DK$ has no boundary segment parallel to $u$. Thus, $\gperim{K}$ is determined on a dense subset of $\real^2$ and, in view of the continuity of $\gperim{K}$ on $DK$ (which follows from Theorem~\ref{g phi thm} \eqref{g phi thm_a}), the covariogram of $\gperim{K}$ is determined on the whole $\real^2$. 

	The second assertion is an immediate consequence of  the first one and of the determination results provided by Theorems~\ref{retrieval centr sym thm}, \ref{retrieval of polygons thm} and \ref{retrieval for width-covar thm}.
\end{proof} 

In order to proceed we need the following lemma. Assume that one does not have access to the \( \phi \)-covariogram directly but only to the \( \phi \)-covariogram scaled by an unknown constant factor. We prove that when \( \phi \in \Phi^2 \setminus \{0\} \) this additional ambiguity is not an obstacle, that is, one can determine the unknown constant factor and by this also the nonscaled \( \phi \)-covariogram.

\begin{lemma} \thmheader{Determination of the multiplicative constant}
	\label{lem mult const}
	Let \( K\in\cK^2_0 \), \( \phi \in \Phi^2 \setminus \{0\} \) and \( \beta > 0 \). Then the knowledge of \( \phi \) and \( \beta \gphi{K} \) determines \( \beta \) and \( \gphi{K} \).
\end{lemma}
\begin{proof}It clearly suffices to determine $\beta$. Let $\phi$ be as in~\eqref{representation_valuation_perim}. Since \( \phi \) is not identically equal to zero, \( \perim_B \) is not identically equal to zero or \( \alpha > 0 \) or both. We introduce parameters \( p,v,c \) as follows:
\begin{equation*}
	p := \perim_B(K),\quad
	v := \vol(K), \quad
	c := \frac{\int_{\real^2} \beta \gphi{K}(x) \dd x}{\beta \gphi{K}(o)}.
\end{equation*}
The parameter \( p \) is determined by the knowledge of \( \beta \gphi{K} \), since Theorem~\ref{g phi thm} (III) yields \( p= \frac{1}{2} \perim_B (\supp (\beta \gphi{K})) \). Furthermore, the parameter \( c \) is determined by \( \beta \gphi{K} \), by construction.

We claim that \( v \) is determined by the knowledge of \( \phi \) and \( c \). By Theorem~\ref{g phi thm} (II) one has
\begin{equation*}
	c = \frac{2 p v + \alpha v^2}{p + \alpha v},
\end{equation*}
which yields
\begin{equation} \label{quad:eq:for:v}
	\alpha v^2 + (2p  - c\alpha ) v - cp  = 0
\end{equation}
In the degenerate case \( \alpha=0 \), we have  \( v = {c}/{2} \) and the claim is proved.
Consider the case \( \alpha>0 \). For a moment, let us view \eqref{quad:eq:for:v} as a quadratic equation in the variable \( v \). Let \( v_1, v_2 \) be the two roots of this equation, counting multiplicities. Note that both roots are real because \( \vol(K) \) is a real root of \eqref{quad:eq:for:v} and thus, the other root is also real. Moreover, by Vieta's formulas \( v_1 v_2 = -{cp}/{\alpha} < 0 \), which shows that one root of \eqref{quad:eq:for:v} is positive and the other one is negative. It follows that \( \vol(K) \) can be determined as the unique positive root of \eqref{quad:eq:for:v}. This concludes the proof of the claim.

Having determined \( p \) and \( v \) we can determine \( \beta \) by the formula
\[
	\beta
		= \frac{\beta \gphi{K}(o)}{\gphi{K}(o)}
		= \frac{\beta \gphi{K}(o)}{p+ \alpha v}.
\]
\end{proof}

In the next theorem we consider a random variable somehow similar to the one studied by Adler and Pyke mentioned above.
Probably the most illustrative case of this random variable is the  one corresponding to \( \beta_1=1 \) and \( \beta_2=0 \), in which case the random variable is associated to the  perimeter-covariogram. 

\begin{theorem} \label{from:rand:var:to:cov} 
	Let \( B\in\cS^2, B \ne \real^2\) and let \( K \in \cK^2_0 \). Let \( X, Z \) and  \( \Sigma \) be mutually independent random variables such that \( \Sigma \) is uniformly distributed in \( \{-1,1\} \) and the densities of \( X \) and \( Z \) coincide, respectively and up to constant multiples, with \( \chf_K \) and \(  \beta_1 \Bdelta_{\bd K}+\beta_2 \chf_K \), where $\beta_1>0$ and $\beta_2 \ge 0$.  Let \( \phi\in\Phi^2 \) be defined by \( \phi= \beta_1 \perim_B+ 2 \beta_2 \vol \). Then the following holds:
	\begin{enumerate}[(I)]
		\item
			The knowledge of $\beta_1,\beta_2, B$ and of the distribution of \( \Sigma(X-Z) \) is equivalent to the knowledge of $\phi$ and the \( \phi \)-covariogram of \( K \).
		\item
			If
			\begin{enumerate}[(a)]
				\item $K$ is centrally symmetric or
				\item $K$ is a polygon and $\beta_2>0$ or 
				\item $K$ is a polygon, $\beta_2=0$ and $B$ is either strictly convex or a strip,
			\end{enumerate}
			then the knowledge of $\beta_1, \beta_2, B$
			and the distribution of \( \Sigma(X-Z) \) determines \( K \), up to translation and reflection, in the class of all planar convex bodies.
	\end{enumerate}
\end{theorem}

\begin{proof}
	Let us prove Assertion (I). The density function of $X$ is $\chf_K/\vol(K)$, while the density of \( Z \) is \( \left(\beta_1 \Bdelta_{\bd K} + \beta_2 \chf_K\right)/c\), where
	\(
		c = \beta_1 \perim_B(K) + \beta_2 \vol(K).
	\)
	Consider a Borel subset \( \Omega \) of \( \real^2. \) Since $\Sigma$ and $X-Z$ are independent and since $\Prob(\Sigma=-1) = \Prob(\Sigma=1)=1/2$, we get
	\begin{align*}
		\Prob(\Sigma(X-Z) \in \Omega )
			& = \frac{1}{2} \big(\Prob(X-Z \in \Omega) + \Prob(Z - X \in \Omega)\big)
			\\ & = \frac{1}{2} \big(\Prob(Z-X \in -\Omega) + \Prob(Z - X \in \Omega)\big).
	\end{align*}
	Thus, the distribution of $\Sigma(X-Z)$ is, up to a multiple, the `even part' of the distribution of $Z-X$. By standard facts in probability, the distribution of $Z-X$ is equal to $\big((\beta_1 \Bdelta_{\bd K} + \beta_2 \chf_K) \ast \chf_{-K} \big)/(c\vol(K))\), i.e. to \( \big(\beta_1\Bdelta_{\bd K} \ast \chf_{-K} + \beta_2\chf_K \ast \chf_{-K}\big)/(c\vol(K))$. By taking the even part of the latter distribution we see that the distribution of $\Sigma(X-Z)$ coincides with
	\[
		\frac{1}{2 c \vol(K)} \bigl( \beta_1 \Bdelta_{\bd K} \ast \chf_{-K} + \beta_1 \Bdelta_{-\bd K} \ast \chf_K + 2 \beta_2 \chf_K \ast \chf_{-K} \bigr ).
	\]
	By Theorem~\ref{g phi thm}~(I), the latter is equal to $\gphi{K}/(2c\vol(K))$.

	Assertion~(I) follows by this and Lemma~\ref{lem mult const}.
	Assertion~(II) is an immediate consequence of Assertion~(I) and of Theorems~\ref{retrieval centr sym thm}, \ref{retrieval of polygons thm} and \ref{retrieval for width-covar thm}.
\end{proof}

\section{Open questions}
\label{sect:open_problems}

\begin{enumerate}[(1)]
\item   Assume that $K$ is a convex polygon. Under which assumptions on the valuation $\phi\in\Phi^2$  does the $\phi$-covariogram problem have a positive answer? And what about the same problem in the case $\phi\notin\Phi^2$, say, if $\phi$ is a continuous translation invariant valuation? See also \cite{MR1837364} for a description of continuous translation invariant valuations in terms of mixed volumes.
 
\item Assume $\phi \in \Phi^2 \setminus \{0\}$ strictly monotone or assume $\phi$ equal to the width in some direction.  Does the $\phi$-covariogram problem has a positive answer for every $K\in\cK^2_0$? 
In the case $\phi=\vol$ the following intermediate question has played an important role in proving a positive answer to this problem.
Assume  $K, H\in\cK^2_0$, $\intr K\cap\intr H\neq\emptyset$ and $\gphi{K}=\gphi{H}$. If $\bd K\cap\bd H$ contains an open arc, is $H=K$? A crucial ingredient in proving a positive answer to this question when $\phi=\vol$ has been a clear geometric interpretation of $\nabla g_K$. 
The  gradient $\nabla g_K(x)$ can be interpreted in terms of the parallelogram inscribed in $K$ and with an edge translate of $x$, and $\nabla g_K=\nabla g_H$   implies that every parallelogram inscribed in $K$ has a translate which is inscribed in $H$.
Thus, it seems interesting to obtain a good understanding of  the information provided by $\nabla \gphi{K}$. 

\item A strengthening of the previous questions is whether the knowledge of $\phi$ is necessary for determination of $K$ from $\gphi{K}$. Formally, this is the question of whether the equality $g_{K,\phi} = g_{H,\psi}$ for $K, H \in \cK_0^2$ and $\phi,\psi \in \Phi^2 \setminus \{0\}$ implies the coincidence of $K$ and $H$, up to translations and reflections. 

\item\label{problem_central_symmetry} Study the $\phi$-covariogram problem when $K$ is a centrally symmetric convex body in $\real^n$, with $n\geq 3$. This problem has certainly a positive answer, for every $n$, when $\phi(K)$ is the surface area of $K$. This generalization can be easily proved following the same lines of the proof of Theorem~\ref{retrieval centr sym thm}. It suffices to extend the representation of the perimeter-covariogram as a convolution to the surface area-covariogram, and to substitute the equality \eqref{censymm3} with the inequality coming from the Brunn-Minkowski inequality for surface area. 
For which quermassintegrals can the problem be treated in the same way?

\item Discussing random variables we noted that $g_K$ is a multiple of the distribution of $X_1-X_2$ for two independent random variables $X_1, X_2$ uniformly distributed in $K$, and so retrieval from $g_K$ can be viewed as the retrieval from the distribution of $X_1-X_2$. In the same vein, for each $K \in \cK_0^n$ one can analyze the information provided by $Y_1-Y_2,$ where $Y_1$ and $Y_2$ are independent random variables uniformly distributed in $\bd K$. Is this information sufficient for determining $K$, up to translations and reflections, when $n=2$? 
This question can be naturally carried over to a more general setting involving arbitrary seminorms (that is, more generally, we can assume that the distributions of $Y_1,Y_2$ coincide with $\Bdelta_{\bd K}/\perim_B$, where $B \in \cS^2$, $B \ne \real^2$).
\end{enumerate}


\def\cprime{$'$}\def\cprime{$'$}\def\cprime{$'$}\def\cprime{$'$}\def\cprime{$'%
$}\def\cprime{$'$}
\providecommand{\bysame}{\leavevmode\hbox to3em{\hrulefill}\thinspace}
\providecommand{\MR}{\relax\ifhmode\unskip\space\fi MR }
\providecommand{\MRhref}[2]{%
  \href{http://www.ams.org/mathscinet-getitem?mr=#1}{#2}
}
\providecommand{\href}[2]{#2}

\end{document}